\documentclass[12pt]{amsart}
\usepackage{a4wide,enumerate}
\usepackage{graphicx}
\usepackage{color}
\allowdisplaybreaks

\usepackage{pgfplots} 
\pgfplotsset{compat=newest}
\pgfplotsset{plot coordinates/math parser=false}

\newlength\densfnfwidth
\newlength\spinfnfwidth
\newlength\temperfnfwidth
\newlength\denslayerswidth
\newlength\spinlayerswidth
\newlength\temperlayerswidth
\let\pa\partial  
\let\na\nabla  
\let\eps\varepsilon  
  
\newcommand{\R}{{\mathbb R}} 
 
\newcommand{\C}{{\mathbb C}} 
\newcommand{\diver}{\operatorname{div}}  

\newcommand{\J}{{\mathcal J}}
\newcommand{\DD}{\mathrm{D}}
\newcommand{\dom}{{\mathcal D}}

\newtheorem{theorem}{Theorem}   
\newtheorem{lemma}[theorem]{Lemma}   
\newtheorem{proposition}[theorem]{Proposition}   
\newtheorem{remark}[theorem]{Remark}

\begin{document}  

\title[Energy-transport models for spin transport]{Energy-transport models for 
spin transport in ferromagnetic semiconductors}

\author{Ansgar J\"ungel}
\address{A.J.: Institute for Analysis and Scientific Computing, Vienna University of  
	Technology, Wiedner Hauptstra\ss e 8--10, 1040 Wien, Austria}
\email{juengel@tuwien.ac.at} 

\author{Polina Shpartko}
\address{P.S.: Institute for Analysis and Scientific Computing, Vienna University of  
	Technology, Wiedner Hauptstra\ss e 8--10, 1040 Wien, Austria}
\email{polina.shpartko@tuwien.ac.at}

\author{Nicola Zamponi}
\address{N.Z.: Institute for Analysis and Scientific Computing, Vienna University of  
	Technology, Wiedner Hauptstra\ss e 8--10, 1040 Wien, Austria}
\email{nicola.zamponi@tuwien.ac.at}

\date{\today}

\thanks{The authors acknowledge partial support from   
the Austrian Science Fund (FWF), grants P22108, P24304, and W1245.} 

\begin{abstract}
Explicit energy-transport equations for the spinorial carrier transport
in ferromagnetic semiconductors are calculated from a general spin energy-transport
system that was derived by Ben Abdallah and El Hajj from a spinorial 
Boltzmann equation. The novelty of our approach are the
simplifying assumptions leading to explicit models which extend both spin
drift-diffusion and semiclassical energy-transport equations.
The explicit models allow us to examine the interplay between the spin and 
charge degrees of freedom. In particular, the monotonicity of the 
entropy (or free energy) and gradient estimates are shown for these models and
the existence of weak solutions to a time-discrete version of one of 
the models is proved, using novel truncation arguments. Numerical experiments
in one-dimensional multilayer structures using a finite-volume discretization 
illustrate the effect of the temperature and the polarization parameter.
\end{abstract}

 

\maketitle

\tableofcontents


\section{Introduction}

Spintronics is a new emerging field in solid-state physics with the aim to
exploit the spin degree of freedom of electrons, which may lead to smaller and
faster semiconductor devices with reduced power consumption. The aim of
the mathematical modeling of spin-polarized materials is to develop a 
hierarchy of models that describe the relevant physical phenomena in an accurate 
way and, at the same time, allow for fast and
efficient numerical predictions. A model class which seems to fulfill the
requirements of precision and simplicity are moment equations derived from
the (spinorial) Boltzmann equation. 

In the literature, up to now,
mostly lowest-order moment equations for spin transport have been investigated,
namely spin drift-diffusion-type equations
\cite{ElH14,PSP05,PoNe11,Sai04}. These models are mathematically analyzed in
\cite{GaGl10,Gli08,JNS15,Zam14}. When hot electron
thermalization has to be taken into account, the carrier transport needs to be
described by higher-order moment equations including energy transport.
This leads to semiclassical energy-transport equations in semiconductors,
see, e.g., \cite{BeDe96,BDG96,CKRSD92,Jue09}. A spinorial energy-transport
model was derived in \cite{BeEl09}, but the equations are not explicit such that
its structure is not easy to analyze.
The goal of this paper is to derive and analyze simplified explicit 
versions of this model.

The starting point is the spinorial Boltzmann equation
for the distribution function $F(x,k,t)$ with values in the space of
Hermitian $2\times 2$ matrices,
\begin{equation}\label{1.be}
  \pa_t F + k\cdot\na_x F - \na_x V\cdot\na_k F = Q(F) 
	+ \frac{\mathrm{i}}{2}[\vec{\Omega}\cdot\vec{\sigma},F] + Q_{\rm sf}(F),
\end{equation}
where $x\in\R^3$ denotes the spatial variable, $k\in\R^3$ the wave vector, 
$t>0$ the time, $\mathrm{i}=\sqrt{-1}$ the imaginary unit, and $[\cdot,\cdot]$
the commutator. The function
$V(x,t)$ is the electric potential, which is usually self-consistently
defined as the solution of the Poisson equation
$$
  -\lambda_D^2\Delta V = n_0[F] - C(x), \quad 
	n_0[F] = \frac12\operatorname{tr}\int_{\R^3}Fdk,
$$
where $\lambda_D$ is the scaled Debye length, $n_0[F]$ the charge density,
``tr'' the trace of a matrix,  
and $C(x)$ the doping concentration \cite{Jue09}. Furthermore,
$\vec{\Omega}(x,k)$ is a local magnetization field and
$\vec{\sigma}=(\sigma_1,\sigma_2,\sigma_3)$ is the vector of the Pauli matrices.
We choose the spin-conserving BGK-type collision operator $Q(F)=M[F]-F$, where the
Maxwellian $M[F]$ is such that $Q(F)$ conserves mass and energy, and
the operator $Q_{\rm sf}(F)$ models spin-flip interactions. 
Details are given in Section \ref{sec.deriv}.

Assuming dominant collisions and a large time scale, moment equations
for the electron density $n[A,C]$ and energy density $W[A,C]$ can be derived
from \eqref{1.be} in the diffusion limit \cite{BeEl09}, leading to
\begin{equation}\label{1.et}
\begin{aligned}
  \pa_t n[A,C] + \diver J_n &= F_n[\vec{\Omega},A,C], \\
	\pa_t W[A,C] + \diver J_W + J_n\cdot\na V &= F_W[\vec{\Omega},A,C],
	\quad x\in\R^3,\ t>0,
\end{aligned}
\end{equation}
where $J_n$ and $J_W$ are the particle and energy flux, respectively, and
$F_n$, $F_W$ are some functions; we refer to Section \ref{sec.deriv} for details.
Furthermore, $A$ and $C$ are the Lagrange
multipliers which are obtained from entropy maximization 
under the constraints of given mass and energy, and
the electron and energy densities are the zeroth- and second-order moments
$$
  n[A,C] = \int_{\R^3}M[A,C]dk, \quad W[A,C] = \frac12\int_{\R^3}M[A,C]|k|^2 dk,
$$
where $M[A,C]=\exp(A+C|k|^2/2)$ is the
spinorial Maxwellian. Note that $A$ and $C$ are Hermitian matrices in
$\C^{2\times 2}$, so $n[A,C]$ and $W[A,C]$ are Hermitian matrices too. 

In contrast to the semiclassical situation, 
the densities cannot be expressed explicitly in terms of the
Lagrange multipliers because of the matrix structure. 
In order to obtain explicit equations, we need to impose 
simplifying assumptions on $A$ and $C$. 
Our strategy is to first formulate the variables in terms 
of the Pauli basis,
$$
  A = a_0\sigma_0 + \vec{a}\cdot\vec{\sigma}, \quad
	C = c_0\sigma_0 + \vec{c}\cdot\vec{\sigma},
$$
where $\sigma_0$ is the unit matrix and $a_0$, $c_0\in\R$, 
$\vec{a}$, $\vec{c}\in\R^3$. The densities may be expanded in
this basis as well, 
$n[A,C] = n_0\sigma_0 + \vec{n}\cdot\vec{\sigma}$, 
$W[A,C] = W_0\sigma_0 + \vec{W}\cdot\vec{\sigma}$,
and the Maxwellian becomes
\begin{equation}\label{1.max}
  M[A,C] = e^{a_0+c_0|k|^2/2}\bigg(\cosh|\vec{b}(k)|\sigma_0
	+ \frac{\sinh|\vec{b}(k)|}{|\vec{b}(k)|}\vec{b}(k)\cdot\vec{\sigma}\bigg),
	\quad \vec{b}(k) := \vec{a}+\vec{c}\frac{|k|^2}{2}.
\end{equation}
The formulation of the energy-transport model \eqref{1.et} in terms of the
Pauli components $(a_0,\vec{a})$, $(c_0,\vec{c})$ still leads to nonexplicit
equations, so we will impose some conditions. We will derive three model
classes by assuming $\vec{c}=0$, $\vec{a}=0$, or $\vec{a}=\lambda\vec{c}$
for some $\lambda=\lambda(x,t)$ and show the following results:

\begin{itemize}
\item First model class ($\vec{c}=0$): 
we discretize the one-dimensional equations using a semi-implicit Euler 
finite-volume scheme
and illustrate the effect of the temperature on two multilayer structures.
\item Second model class ($\vec{a}=0$): we show the existence of weak solutions 
to a time-discrete version. 
\item Third model class ($\vec{a}=\lambda\vec{c}$): we show that the equation
for the spin accumulation density $\vec{s}=\vec{n}/|\vec{n}|$ has some similarities
with the Landau-Lifshitz equation.
\item All model classes: we compute the entropy (free energy) and the
entropy production, thus providing not only the monotonicity of the entropy
but also gradient estimates.
\end{itemize}

These findings are a first step to understand higher-order spinorial
macroscopic models which may lead to improved simulation outcomes.

The paper is organized as follows. The main results are detailed in Section
\ref{sec.main}. The derivation of the general energy-transport
model from the spinorial Boltzmann equation is recalled in Section 
\ref{sec.deriv} and the general model is formulated in terms of the Pauli components
in Section \ref{sec.pauli}. In Section \ref{sec.simpl}, the three simplified 
energy-transport model classes are derived. 
The entropy structure is investigated in Section \ref{sec.ent}, and the
existence result for the second model is stated and proved in Section 
\ref{sec.ex}. Some numerical experiments for the first model are performed
in Section \ref{sec.num}.


\section{Main results}\label{sec.main}

We detail the main results of this paper.

\subsection{Derivation of explicit spin energy-transport models}
We derive explicit versions of \eqref{1.et} under three simplifying 
assumptions on the Pauli components of $A$ and $C$.

\subsubsection*{First model: $\vec{c}=0$.} 
If the Lagrange multiplier $C$
is interpreted as a ``temperature'' tensor, it might be reasonable to
suppose that the ``spin'' part $\vec{c}$ is much smaller than the 
non-vanishing trace part $c_0$, which motivates the simplification $\vec{c}=0$.
This allows us to write three of the eight scalar moments $(n_0,\vec{n})$ and
$(W_0,\vec{W})$ in terms of the remaining moments, leading to equations
for five moments. We choose the moments $(n_0,\vec{n},W_0)$, leading to 
the system (see Section \ref{sec.et1})
\begin{align}
  & \pa_t n_0 + \diver J_n = 0, \quad J_n = -\big(\na(n_0T) + n_0\na V\big), 
	\label{1.et1.n0} \\
	& \frac32\pa_t(n_0T) + \diver J_W + J_n\cdot\na V = 0, \quad
	J_W = -\frac52\big(\na(n_0T^2) + n_0T\na V\big), \label{1.et1.n0T} \\
	& \pa_t \vec{n} - \sum_{j=1}^3\pa_{x_j}\big(\pa_{x_j}
	(\vec{n}T)+\vec{n}\pa_{x_j}V\big)
	+ \vec{\Omega}_{\rm e}\times\vec{n} = -\frac{\vec{n}}{\tau_{\rm sf}},
	\quad x\in\R^3,\ t>0, \label{1.et1.vecn}
\end{align}
where $T=2W_0/(3n_0)$ is interpreted as the electron temperature, 
$\pa_{x_j}=\pa/\pa x_j$, $\vec{\Omega}_{\rm e}$ is the even part of the
effective field (with respect to $k$), and
$\tau_{\rm sf}>0$ is the spin-flip relaxation time. 
In this model, $(n_0,\frac32 n_0T_0)$ solves the semiclassical energy-transport
equations, and the spin-vector density $\vec{n}$ solves a drift-diffusion-type
equation, which is coupled to the equations for $(n_0,\frac32 n_0T_0)$ via $T$ only.
Our numerical experiments indicate that this coupling is rather weak.

Motivated from \cite{PoNe11}, we may include a polarization matrix $P$ in the
definition of the collision operator $Q(F)$. We choose $Q(F)=P^{1/2}(M[F]-F)P^{1/2}$,
where the direction of $P=\sigma_0+p\vec{\Omega}\cdot\vec{\sigma}$ in spin space
is the local magnetization $\vec\Omega$ and $p\in[0,1)$ represents the spin 
polarization of the scattering rates. This operator conserves spin, mass,
and (in contrast to the operators in \cite{PoNe11}) energy.
The corresponding spin energy-transport model (still under the assumption 
$\vec{c}=0$) becomes (see Remark \ref{rem.et1})
\begin{align}
  & \pa_t n_0 + \diver\J_n = 0, \quad 
	\J_n = \eta^{-2}\big(J_n - p\vec{\Omega}\cdot\vec{J}_n\big), \label{1.et1a.n0} \\
	& \frac32\pa_t(n_0T) + \diver\J_W + \J_n\cdot\na V=0, \quad
	\J_W = \eta^{-2}\big(J_W - p\vec{\Omega}\cdot\vec{J}_W\big), \label{1.et1a.n0T} \\
	& \pa_t\vec{n} + \diver\vec{\J} + \vec{\Omega}_{\rm e}\times\vec{n}
	= -\frac{\vec{n}}{\tau_{\rm sf}}, \quad x\in\R^3,\ t>0, \label{1.et1a.vecn}
\end{align}
where $\eta=\sqrt{1-p^2}$, $J_n$, $J_W$ are as above, and
\begin{align*}
  \vec{J}_n &= -\big(\na(\vec{n}T) + \vec{n}\na V\big), \quad
  \vec{J}_W = -\frac52\big(\na(\vec{n}T^2) + \vec{n}\na V\big), \\
  \vec{\J} &= \eta^{-2}\big((1-\eta)(\vec{J}_n\cdot\vec{\Omega})\vec{\Omega}
	+ \eta\vec{J}_n - p\vec{\Omega}J_n \big).
\end{align*}
Note that we recover the model \eqref{1.et1.n0}-\eqref{1.et1.vecn} if $p=0$.
We compare both models numerically in Section \ref{sec.num}.
It turns out that the polarization matrix $P$ leads to a stronger mixing of
the spin density components, and the heat flux effects causes a smoothing of
these components.

\subsubsection*{Second model: $\vec{a}=0$.} 
The Lagrange multiplier $A$ may be related to the particle density.
Supposing that the spin effects are rather encoded in $\vec{c}$, one may
assume that $\vec{a}=0$. This condition gives as above three constraints and leads
to equations for five moments. One may choose, for instance, the variables
$(n_0,\vec{n},T)$ or $(n_0,T,\vec{W})$. In the former case,
we arrive at the system of coupled equations
\begin{align}
  & \pa_t n_0 + \diver J_n = 0, \quad J_n = -\big(\na(n_0T) + n_0\na V\big), 
	\label{1.et2.n0} \\
  & \frac32\pa_t(n_0T) + \diver J_W + J_n\cdot\na V = 0, \quad
	J_W = -\frac52\big(\na(D(n_+,n_-)n_0T^2) + n_0T\na V\big), \label{1.et2.n0T} \\
	& \pa_t\vec{n} - \sum_{j=1}^3\pa_{x_j}\bigg(\pa_{x_j}
	\bigg(p(n_+,n_-)n_0T\frac{\vec{n}}{|\vec{n}|}\bigg)+\vec{n}\pa_{x_j}V\bigg)
	+ \vec{\Omega}_{\rm e}\times\vec{n} = -\frac{\vec{n}}{\tau_{\rm sf}},
	\label{1.et2.vecn}
\end{align}
and $D(n_+,n_-)$, $p(n_+,n_-)$, defined in \eqref{32.Dp}, depend on the
spin-up/spin-down densities $n_\pm:=n_0\pm|\vec{n}|$ (see Section \ref{sec.et2}).
Compared to the first model, these coefficients realize a coupling between the
charge and spin-vector densities. A similar model can be derived in the variables
$(n_0,T,\vec{W})$. This coupling is still rather weak
since the function $D(n_+,n_-)$ only takes values in the interval $[1,1.1]$; see
Remark \ref{rem.D}.

\subsubsection*{Third model: $\vec{a}=\lambda\vec{c}$.} 
Generalizing the above approaches, we suppose that the vectors
$\vec{a}$ and $\vec{c}$ are aligned such
that $\vec{a}=\lambda\vec{c}$ for some function $\lambda=\lambda(x,t)\neq 0$.
The first model is recovered for $\lambda\to \infty$, the second one for
$\lambda=0$. This condition
provides only two constraints such that we obtain a system for six moments.
A possible choice is $(n_\pm,W_\pm,\vec{s})$, where
$n_\pm=n_0\pm|\vec{n}|$, $W_\pm=W_0\pm|\vec{W}|$, and $\vec{s}=\vec{n}/|\vec{n}|$,
which gives the equations
\begin{align}
  & \pa_tn_\pm + \diver J_{n,\pm} = \mp\frac{n_+ - n_-}{2\tau_{\rm sf}}
	\mp \frac12(n_+T_+ - n_-T_-)|\na\vec{s}|^2, \label{1.et3.npm} \\
  & \frac32\pa_t(n_\pm T_\pm) + \diver J_{W,\pm} + J_{n,\pm}\cdot\na V
	= \mp\frac{3}{4\tau_{\rm sf}}(n_+T_+ - n_-T_-) \label{1.et3.nT} \\
	&\phantom{\frac32\pa_t(n_\pm T_\pm) + \diver J_{W,\pm} + J_{n,\pm}\cdot\na V
	=}{}\mp \frac54(n_+T_+^2 - n_-T_-^2)|\na\vec{s}|^2, \nonumber \\
	& \pa_t\vec{s} - \frac{n_+T_+ - n_-T_-}{n_+ - n_-}\vec{s}\times
	(\Delta\vec{s}\times\vec{s})
	= \bigg(2\frac{\na(n_+T_+ - n_-T_-)}{n_+ - n_-} + \na V\bigg)\cdot\na\vec{s}
	- \vec{\Omega}_{\rm e}\times\vec{s}, \label{1.et3.s}
\end{align}
with the spin-up/spin-down particle and heat fluxes
\begin{equation}\label{1.et3.J}
  J_{n,\pm} = -\big(\na(n_\pm T_\pm) + n_\pm\na V\big), \quad
	J_{W,\pm} = -\frac52\big(\na(n_\pm T_\pm^2) + n_\pm T_\pm\na V\big),
\end{equation}
and the spin-up/spin-down energy densities $W_\pm = \frac32n_\pm T_\pm$.
The evolution equations for the spin-up/spin-down densities 
are similar in structure as the first and second model.
For constant ``temperature'' $T_+=T_-=1$, we recover the two-component
spin drift-diffusion equations analyzed in \cite{Gli08}.
The coupling is realized through the spin-accumulation density
$\vec{s}$. The equation for $\vec{s}$ preserves the relation $|\vec{s}|=1$, and
the second-order term
$\vec{s}\times(\Delta\vec{s}\times\vec{s})$ also appears in the
Landau-Lifshitz equation \cite{LaLi35}; see Remark \ref{rem.LLG}.

\subsection{Entropy inequalities}

We prove that there exists an entropy
(or free energy) which is nonincreasing in time along solutions to the
corresponding equations.\footnote{In contrast to the physical notation, 
the mathematical entropy is defined here as the negative physical entropy.}
To simplify the computations, we neglect electric effects, 
i.e., the potential $V$ is assumed to be constant (also see Remark \ref{rem.ent1}
for the general situation). 

The kinetic entropy of the general spin model \eqref{1.et} is given by
\begin{equation}\label{1.H}
  H = \int_{\R^3}\int_{\R^3}\operatorname{tr}(M\log M)dkdx,
\end{equation}
where the Maxwellian is defined by \eqref{1.max} and ``tr'' denotes the 
trace of a matrix. It was shown in \cite[Theorem 2.2]{BeEl09} 
that the entropy is nonincreasing along solutions to \eqref{1.et}.
Our aim is to quantify the entropy production $-dH/dt$
which provides gradient estimates. 
To this end, we insert the simplifying Maxwellians in \eqref{1.H} and
compute explicit expressions for the entropies. Denoting by $H_j$ the entropy
of the $j$th model presented above, we obtain
\begin{align}
  H_1 &= \int_{\R^3}\big(n_+\log(n_+T_+^{-3/2}) + n_-\log(n_-T_-^{-3/2})\big)dx, 
	\label{1.H1} \\
	H_2 &= \frac52\int_{\R^3}n_0\log\frac{n_0}{W_+^{3/5}+W_-^{3/5}}dx,
	\quad\mbox{where }W_\pm=\frac32n_0T\pm |\vec{W}|^2, \label{1.H2} \\
	H_3 &= H_1, \label{1.H3}
\end{align}
and the corresponding entropy inequalities read as (see Propositions 
\ref{prop.ent1}-\ref{prop.ent3})
\begin{align}
  \frac{dH_1}{dt} + 4\int_{\R^3}\big(|\na\sqrt{n_+ T}|^2 + |\na\sqrt{n_- T}|^2
	+ 5n_0|\na\sqrt{T}|^2\big)dx &\le 0, \nonumber \\
	\frac{dH_2}{dt} + c\int_{\R^3}\big(|\na \sqrt{W_+}|^2 + |\na \sqrt{W_-}|^2
	+ T|\na \sqrt{n_0}|^2\big) dx &\le 0, \label{1.dhdt2} \\
  \frac{dH_3}{dt} + c\int_{\R^3}\sum_{s=\pm}\big(T_s|\na\sqrt{n_s}|^2
	+ n_s|\na\sqrt{T_s}|^2\big)dx &\le 0, \nonumber
\end{align}
where $c>0$ is some number and the results hold for smooth solutions.

\subsection{Existence analysis for the second model}

The second analytical result concerns the existence analysis for the 
second model ($\vec{a}=0$) in the variables $(n_0,W_0,\vec{W})$, where
$W_0=\frac32n_0T$. Because
of the strong coupling, we are only able to prove the existence of solutions
to a time-discrete version without electric field in a bounded domain 
$\dom\subset\R^d$: 
\begin{align}
  \frac{1}{h}(n_0-n_0^0) - \frac23\Delta W_0 &= 0, \label{12.n0} \\
	\frac{1}{h}(W_0-W_0^0) - \frac{8}{15}\Delta\bigg(\frac{1}{n_0}
	(W_+^{3/5}+W_-^{3/5})(W_+^{7/5}+W_-^{7/5})\bigg) &= 0, \label{12.w0} \\
	\frac{1}{h}(\vec{W}-\vec{W}^0) - \frac{5}{18}\Delta\bigg(\frac{1}{n_0}
	(W_+^{3/5}+W_-^{3/5})(W_+^{7/5}-W_-^{7/5})\frac{\vec{W}}{|\vec{W}|}\bigg)
	&= -\frac{\vec{W}}{\tau_{\rm sf}}\quad \mbox{in }\dom, \label{12.vecw}
\end{align}
where $(n_0,W_0,\vec{W})$ is the solution at the actual time, 
$(n_0^0,W_0^0,\vec{W}^0)$ is the solution at the previous time instant, and
$h>0$ is the time step size;
see Theorem \ref{thm.ex}. The boundary conditions are given by
\begin{equation}\label{12.bc}
  n_0 = n_0^D, \quad W_0 = W_0^D, \quad \vec{W}=\vec{W}^D\quad\mbox{on }\pa\dom.
\end{equation}

The main difficulty in the existence proof is the derivation of 
suitable a priori estimates. The entropy-production inequality \eqref{1.dhdt2}
provides estimates which are uniform in $h$, but they are not sufficient to
pass to the limit $h\to 0$ since the gradient estimate \eqref{1.dhdt2}
for $\na\sqrt{n_0}$ becomes useless in regions where $T$ is close to zero.
 
Our proof employs some ideas from \cite{ZaJu15}.
The first idea is to formulate \eqref{12.n0}-\eqref{12.vecw} as
\begin{align*}
  n_0(u,v_0,\vec{v})-n_0^0 & = h\Delta u, \\
	W_0(u,v_0,\vec{v})-W_0^0 & = h\Delta v_0, \\
	\vec{W}(u,v_0,\vec{v})-\vec{W}^0 & = h\Delta \vec{v} 
	- (h/\tau_{\rm sf})\vec{W} \quad\mbox{in }\dom,
\end{align*}
where $(u,v_0,\vec{v})$ are some auxiliary variables. The second idea is
to truncate the new variables by replacing $u$ by $[u/v_0]_\eps v_0$, where
$[\cdot]_\eps$ is a truncation operator satisfying $[u/v_0]_\eps v_0=u$
for $0<u/v_0\le 1/\eps$. The existence of weak solutions
to the truncated problem is shown by means of the Leray-Schauder fixed-point
theorem. The compactness follows from standard $H^1$ elliptic estimates.
Then, choosing special Stampacchia-type test functions, we prove lower and
upper bounds for the new variables, which allow us to remove the truncation.
In this step, we exploit the particular structure of the equations.

Unfortunately, our a priori estimates depend on the time step size which
prevents the limit $h\to 0$. Even the analysis of the time-discrete equations
is highly delicate since the equations are elliptic in a non-standard sense.
The existence of weak solutions to the semiclassical energy-transport 
equations near equilibrium was proved in \cite{ChHs03,FaIt01,Gri99}.
An existence analysis for general initial data was shown in \cite{DGJ97}
but for uniformly positive definite diffusion matrices only.
A semiclassical energy-transport system without electric effects has been
investigated in \cite{ZaJu15}. This system possesses similar difficulties 
as \eqref{12.n0}-\eqref{12.vecw} but its structure is easier. 
For details, we refer to Section \ref{sec.ex}. 


\section{A general energy-transport model for spin transport}

\subsection{Derivation from the spinorial Boltzmann equation}\label{sec.deriv}

We sketch briefly the deri\-vation of the general energy-transport model \eqref{1.et}
from the spinorial Boltzmann transport equation \eqref{1.be}. Details are
given in \cite{BeEl09}. We consider the Boltzmann equation in the diffusion scaling,
\begin{equation}\label{2.be}
  \pa_t F_\eps + \frac{1}{\eps}\big(k\cdot\na_x F_\eps - \na_x V\cdot\na_k F_\eps\big)
	= \frac{1}{\eps^2}Q(F_\eps) 
	+ \frac{\mathrm{i}}{2}[\vec\Omega_\eps(x,k)\cdot\vec\sigma,F_\eps]
	+ Q_{\rm sf}(F_\eps),
\end{equation}
The parameter $\eps>0$ is the scaled mean free path and is supposed to be small.
We have assumed the parabolic-band approximation such that the mean velocity
equals $v(k)=k$. 

The last term in \eqref{2.be} represents the spin-flip interactions 
which are specified in \eqref{3.sf} below.
The commutator $[\cdot,\cdot]$ on the right-hand side
of \eqref{2.be} can be rigorously derived
from the Schr\"odinger equation with spin-orbit Hamiltonian in the semiclassical
limit \cite[Chapter~1]{ElH08}. The term models a precession effect around
the effective field \cite{BeEl09}. 

The first term on the right-hand side of \eqref{2.be} models
collisions that conserve mass and energy. For simplicity, we employ
the BGK-type operator (named after Bhatnagar, Gross, and Krook)
$Q(F) = M[F]-F$,
where the Maxwellian $M[F]$ associated to $F$ has the same mass and energy as $F$,
\begin{equation}\label{2.AC}
  \int_{\R^3}M[F]dk = \int_{\R^3}Fdk, \quad 
	\frac12\int_{\R^3}M[F]|k|^2dk = \frac12\int_{\R^3}F|k|^2dk.
\end{equation}
The Maxwellian is constructed from entropy maximization under the constraints of 
given mass and energy, which yields, in case of
Maxwell-Boltzmann statistics, the existence of Lagrange multipliers
$A(x,t)$ and $C(x,t)$ such that \cite{BeEl09}
$$
  M[F](x,k,t) = \exp\bigg(A(x,t)+C(x,t)\frac{|k|^2}{2}\bigg),
$$
where exp is the matrix exponential and 
$A$, $C$ are Hermitian $2\times 2$ matrices satisfying 
\eqref{2.AC}. 

The space of Hermitian $2\times 2$ matrices can be spanned by the
unit matrix $\sigma_0$ and the Pauli matrices 
$\vec\sigma=(\sigma_1,\sigma_2,\sigma_3)$,
$$
  \sigma_0 = \begin{pmatrix} 1 & 0 \\ 0 & 1 \end{pmatrix}, \quad
  \sigma_1 = \begin{pmatrix} 0 & 1 \\ 1 & 0 \end{pmatrix}, \quad	
  \sigma_2 = \begin{pmatrix} 0 & -\mathrm{i} \\ \mathrm{i} & 0 \end{pmatrix}, \quad
  \sigma_3 = \begin{pmatrix} 1 & 0 \\ 0 & -1 \end{pmatrix}.
$$
Accordingly, we may write $A=a_0+\vec{a}\cdot\vec\sigma$ and
$C=c_0+\vec{c}\cdot\vec\sigma$, where $a_0$, $c_0\in\R$,
$\vec{a}=(a_1,a_2,a_3)$, $\vec{c}=(c_1,c_2,c_3)\in\R^3$, and
$\vec{a}\cdot\vec{\sigma}=\sum_{j=1}^3 a_j\sigma_j$. The coefficients
in the Pauli basis are computed from $a_0=\frac12\operatorname{tr}(A)$,
$\vec{a}=\frac12\operatorname(\vec{\sigma}A)$, and similarly for $c_0$,
$\vec{c}$; see, e.g., \cite{PoNe11}. The matrix exponential can be also
expanded in the Pauli matrix, giving 
$M[F] = M_0\sigma_0 + \vec{M}\cdot\vec\sigma$, where
\begin{equation}\label{2.MPauli}
  M_0 = e^{a_0+c_0|k|^2/2}\cosh\bigg|\vec{a}+\vec{c}\frac{|k|^2}{2}\bigg|, \quad
	\vec{M} = e^{a_0+c_0|k|^2/2}\frac{\sinh
	|\vec{a}+\vec{c}|k|^2/2|}{|\vec{a}+\vec{c}|k|^2/2|}
	\bigg(\vec{a}+\vec{c}\frac{|k|^2}{2}\bigg).
\end{equation}

It is shown in \cite[Theorem 3.1]{BeEl09} that $F_\eps$ converges formally
to $M:=M[A,C]=\exp(A+C|k|^2/2)$ as $\eps\to 0$, 
where $(A,C)$ are solutions to the following
spin energy-transport system for the electron density $n(x,t)$ and energy
density $W(x,t)$, which are related to $(A,C)$ via the moment equations
$$
  n = \int_{\R^3}M[A,C]dk, \quad
	W = \frac12\int_{\R^3}M[A,C]|k|^2 dk.
$$
The general energy-transport equations read as \cite[Theorem~3.1]{BeEl09}
\begin{align}
  & \pa_t n + \diver_x J_n = \frac{\mathrm{i}}{2}\int_{\R^3}\big|\vec{\Omega}_{\rm ET}
	\cdot\vec\sigma,M\big]dk - \frac14\int_{\R^3}\big[\vec{\Omega}_{\rm o}
	\cdot\vec\sigma,[\vec{\Omega}_{\rm o}\cdot\vec\sigma,M]\big]dk \label{2.etn} \\
	&\phantom{xx}{}+ \int_{\R^3}Q_{\rm sf}(M)dk, \nonumber \\
	& \pa_t W + \diver_x J_W + J_n\cdot\na_x V
	= \frac{\mathrm{i}}{2}\int_{\R^3}\big[\vec{\Omega}_{\rm o}\cdot\vec\sigma,M\big]
	\frac{|k|^2}{2}dk \label{2.etw} \\
	&\phantom{xx}{}- \frac14\int_{\R^3}\big[\vec{\Omega}_{\rm o}
	\cdot\vec\sigma,[\vec{\Omega}_{\rm o}\cdot\vec\sigma,M]\big]\frac{|k|^2}{2}dk
	+ \frac12\int_{\R^3}Q_{\rm sf}(M)|k|^2dk,
	\nonumber
\end{align}
where the effective field $\vec{\Omega}_{\rm ET}$ is defined by
\begin{equation}\label{2.OmET}
  \vec{\Omega}_{\rm ET} = (k\cdot\na_x-\na_x V\cdot\na_k)\vec{\Omega}_{\rm o}
	+ \vec{\Omega}_{\rm e},
\end{equation}
and $\vec{\Omega}_{\rm o}$ and $\vec{\Omega}_{\rm e}$ are the odd and even
parts of $\vec{\Omega}$ (with respect to $k$), respectively.
The tensor-valued fluxes are defined by
\begin{equation}\label{2.J}
\begin{aligned}
  J_n &= -\diver_x\Pi - n\na_x V + \Pi_{\Omega_{\rm o}}, \\
	J_W &= -\diver_x Q - (W\sigma_0+\Pi)\na_x V + Q_{\Omega_{\rm o}},
\end{aligned}
\end{equation}
and the tensors $\Pi=(\Pi^{j\ell})$, $Q=(Q^{j\ell})$ with
$\Pi^{j\ell}$, $Q^{j\ell}\in\C^{2\times 2}$ and 
$\Pi_{\Omega_{\rm o}}=(\Pi_{\Omega_{\rm o}}^j)$, 
$Q_{\Omega_{\rm o}}=(Q_{\Omega_{\rm o}}^j)$ with 
$\Pi_{\Omega_{\rm o}}^j$, $Q_{\Omega_{\rm o}}^j\in\C^{2\times 2}$
are given by the moments
\begin{align*}
  \Pi^{j\ell} &= \int_{\R^3}k_j k_\ell Mdk, & 
	Q^{j\ell} &= \frac12\int_{\R^3}k_j k_\ell|k|^2Mdk, \\
	\Pi_{\Omega_{\rm o}}^j 
	&= \mathrm{i}\int_{\R^3}\big[\vec{\Omega}_{\rm o}\cdot\vec\sigma,M\big]k_j dk, & 
	Q_{\Omega_{\rm o}}^j &= \frac{\mathrm{i}}{2}\int_{\R^3}
	\big[\vec{\Omega}_{\rm o}\cdot\vec\sigma,M\big]k_j|k|^2dk,
\end{align*}
where $j,\ell=1,2,3$.
The first two terms on the right-hand sides of \eqref{2.etn} and \eqref{2.etw} 
are due to spinor effects; they vanish
in the classical energy-transport model. The last term on the left-hand side
of \eqref{2.etw} is the Joule heating and it is also present in the classical
model. The last terms in \eqref{2.etn}-\eqref{2.etw} express the moments
of the spin-flip interactions. 


\subsection{Formulation in the Pauli basis}\label{sec.pauli}

In order to derive simplified spin energy-transport models in explicit form,
it is convenient to formulate \eqref{2.etn}-\eqref{2.etw} in the Pauli basis.
Recall that $n = n_0\sigma_0+\vec{n}\cdot\vec\sigma$ and
$W=W_0\sigma_0+\vec{W}\cdot\vec\sigma$. Furthermore, we expand
\begin{equation}\label{2.Q}
  \int_{\R^3}Q_{\rm sf}(M)dk
  = Q_{{\rm sf},n,0}\sigma_0+\vec{Q}_{{\rm sf},n}\cdot\vec\sigma, \quad
  \frac12\int_{\R^3}Q_{\rm sf}(M)|k|^2dk
  = Q_{{\rm sf},W,0}\sigma_0+\vec{Q}_{{\rm sf},W}\cdot\vec\sigma.
\end{equation}

\begin{lemma}[Energy-transport model in Pauli components]\label{lem.fullet}\sloppy
Equations \eqref{2.etn}-\eqref{2.etw} can be written in the Pauli components
$(n_0,\vec{n})$ and $(W_0,\vec{W})$ as
\begin{align}
  & \pa_t n_0 - \diver_x\bigg(\frac23\na_x W_0+n_0\na_x V\bigg) = Q_{{\rm sf},n,0}, 
	\label{2.etn0} \\
	& \pa_t\vec{n} - \sum_{j=1}^3\pa_{x_j}\bigg(\frac23\pa_{x_j}\vec{W} 
	+ \vec{n}\pa_{x_j} V
	+ 2\int_{\R^3}(\vec{\Omega}_{\rm o}\times\vec{M})k_jdk\bigg) \label{2.etvecn} \\
	&\phantom{xx}{}+ \sum_{j=1}^3\int_{\R^3}\pa_{x_j}
	(\vec{\Omega}_{\rm o}\times\vec{M})k_jdk
	+ \sum_{j=1}^3\pa_{x_j} V\int_{\R^3}\pa_{k_j}(\vec{\Omega}_{\rm o}\times\vec{M})dk
	+ \int_{\R^3}\vec{\Omega}_{\rm e}\times\vec{M}dk \nonumber \\
	&\phantom{xx}{}+ \int_{\R^3}\big(|\vec{\Omega}_{\rm o}|^2 
	- \vec{\Omega}_{\rm o}\otimes\vec{\Omega}_{\rm o}\big)\cdot\vec{M}dk 
	= \vec{Q}_{{\rm sf},n}, \nonumber \\
	& \pa_t W_0 - \diver_x\bigg(\frac16\int_{\R^3}\na_x M_0|k|^4 dk
	+ \frac53 W_0\na_x V\bigg) \label{2.etw0} \\
	&\phantom{xx}{}- \bigg(\frac23\na_x W_0+n_0\na_x V\bigg)\cdot\na_x V
	= Q_{{\rm sf},W,0}, \nonumber \\
	& \pa_t\vec{W} - \sum_{j=1}^3\bigg\{\pa_{x_j}\bigg(\frac16\pa_{x_j}\int_{\R^3}
	\vec{M}|k|^4dk + \frac53\vec{W}\pa_{x_j}V 
	+ \int_{\R^3}(\vec{\Omega}_{\rm o}\times\vec{M})k_j|k|^2dk\bigg) \label{2.etvecw} \\
	&\phantom{xx}{}+ \bigg(\frac23\pa_{x_j}\vec{W} + \vec{n}\pa_{x_j}V
	+ 2\int_{\R^3}(\vec{\Omega}_{\rm o}\times\vec{M})k_jdk\bigg)\pa_{x_j} V\bigg\} 
	\nonumber \\
	&\phantom{xx}{}+ \frac12\sum_{j=1}^3\pa_{x_j}
	\int_{\R^3}(\vec{\Omega}_{\rm o}\times\vec{M})k_j|k|^2dk
	+ \frac12\sum_{j=1}^3\pa_{x_j}V\int_{\R^3}
	\pa_{k_j}(\vec{\Omega}_{\rm o}\times\vec{M})|k|^2 dk \nonumber \\
	&\phantom{xx}{}+ \frac12\int_{\R^3}(\vec{\Omega}_{\rm e}\times\vec{M})|k|^2dk
	+ \int_{\R^3}\big(|\vec{\Omega}_{\rm o}|^2 
	- \vec{\Omega}_{\rm o}\otimes\vec{\Omega}_{\rm o}\big)\cdot\vec{M}|k|^2 dk 
	= \vec{Q}_{{\rm sf},W}, \nonumber
\end{align}
where $\pa_{x_j}=\pa/\pa x_j$, $\pa_{k_j}=\pa/\pa k_j$.
\end{lemma}

\begin{proof}
We reformulate \eqref{2.etn}-\eqref{2.etw} in terms of the Pauli coefficients.
For this, set $J_{n}=(J_{n}^j)_{j=1,2,3}$, $J_W=(J_W^j)_{j=1,2,3}$ and
$J_n^j=J_{n,0}^j\sigma_0 + \vec{J}_n\cdot\vec{\sigma}$, 
$J_W^j=J_{W,0}^j\sigma_0 + \vec{J}_W\cdot\vec{\sigma}$.
We obtain
\begin{align}
  & \pa_t n_0 + \sum_{j=1}^3\pa_{x_j}J_{n,0}^j = Q_{{\rm sf},n,0}, \label{2.aux1} \\
	& \pa_t\vec{n} + \sum_{j=1}^3\pa_{x_j}\vec{J}_n^j 
	+ \int_{\R^3}\vec{\Omega}_{\rm ET}\times\vec{M} dk
	+ \int_{\R^3}\big(|\vec{\Omega}_{\rm o}|^2 
	- \vec{\Omega}_{\rm o}\otimes\vec{\Omega}_{\rm o}\big)\cdot\vec{M}dk 
	= \vec{Q}_{{\rm sf},n}, \label{2.aux2} \\
	& \pa_t W_0 + \sum_{j=1}^3\big(\pa_{x_j}J_{W,0}^j + J_{n,0}^j\pa_{x_j}V\big) 
	= Q_{{\rm sf},W,0}, \label{2.aux3} \\
	& \pa_t\vec{W} + \sum_{j=1}^3\big(\pa_{x_j}\vec{J}_W^j + \vec{J}_n^j\pa_{x_j}V\big)
	+ \int_{\R^3}(\vec{\Omega}_{\rm ET}\times\vec{M})|k|^2 dk \label{2.aux4} \\
	&\phantom{xxm}+ \frac12\int_{\R^3}\big(|\vec{\Omega}_{\rm o}|^2 
	- \vec{\Omega}_{\rm o}\otimes\vec{\Omega}_{\rm o}\big)\cdot\vec{M}|k|^2 dk 
	= \vec{Q}_{{\rm sf},W}. \nonumber
\end{align}
Let us expand the integrals involving $\vec{\Omega}_{\rm ET}$
and the fluxes. Let $\phi(k)=1$ or $\phi(k)=|k|^2/2$. Then, recalling
definition \eqref{2.OmET} for $\vec{\Omega}_{\rm ET}$,
\begin{align*}
  \int_{\R^3} & (\vec{\Omega}_{\rm ET}\times\vec{M})\phi(k)dk \\
	&= \int_{\R^3}\big(k\cdot\na_x - \na_x V\cdot\na_k\big)(\vec{\Omega}_{\rm o}
	\times\vec{M})\phi(k)dk
	+ \int_{\R^3}(\vec{\Omega}_{\rm e}\times\vec{M})\phi(k)dk \\
	&= \sum_{j=1}^3\pa_{x_j}\int_{\R^3}(\vec{\Omega}_{\rm o}\times\vec{M})k_j\phi(k)dk
	- \sum_{j=1}^3\pa_{x_j}V\int_{\R^3}(\pa_{k_j}\vec{\Omega}_{\rm o}\times\vec{M})
	\phi(k)dk \\
	&\phantom{xx}{}+ \int_{\R^3}(\vec{\Omega}_{\rm e}\times\vec{M})\phi(k)dk.
\end{align*}
Inserting these expressions into the evolution equations for $\vec{n}$ and $\vec{W}$,
we recover the three integrals in the second line of \eqref{2.etvecn} as well as
the integrals in the third line, and the first integral in the 
fourthline of \eqref{2.etvecw}.

It remains to compute the fluxes \eqref{2.J}. First, we calculate
$$
  \Pi^{j\ell} = \frac13\int_{\R^3}M|k|^2 dk\delta_{j\ell} = \frac23W\delta_{j\ell},
	\quad Q^{j\ell} = \frac16\int_{\R^3}M|k|^4dk\delta_{j\ell}.
$$
Furthermore, using the formula $[\vec{u}\cdot\vec{\sigma},\vec{v}\cdot\vec{\sigma}]
= 2\mathrm{i}(\vec{u}\times\vec{v})\cdot\vec{\sigma}$ for $\vec{u}$, $\vec{v}\in\R^3$,
we find that 
$\Pi_{\Omega_{\rm o}}=\Pi_{\Omega_{\rm o},0}\sigma_0+\vec{\Pi}_{\Omega_{\rm o}}
\cdot\vec\sigma$ with $\Pi_{\Omega_{\rm o},0}=0$ and 
$\vec{\Pi}_{\Omega_{\rm o}}=-2\int_{\R^3}(\vec{\Omega}_{\rm o}\times\vec{M})k dk$.
Therefore,
\begin{align*}
  J_{n,0}^j &= -\sum_{\ell=1}^3\pa_{x_\ell}\Pi_{0}^{j\ell}- n_0\pa_{x_j}V
	+ \Pi_{\Omega_0,0}^j
	= -\frac23\pa_{x_j}W_0 - n_0\pa_{x_j}V, \\
  \vec{J}_{n}^j &= -\sum_{\ell=1}^3\pa_{x_\ell}\vec{\Pi}^{j\ell}
	- \vec{n}\pa_{x_j}V + \vec{\Pi}_{\Omega_{\rm o}}^j
	= -\frac23\pa_{x_j}\vec{W} - \vec{n}\pa_{x_j}V 
	- 2\int_{\R^3}(\vec{\Omega}_{\rm o}\times\vec{M})k_j dk.
\end{align*}
Expanding $Q_{\Omega_{\rm o}}=Q_{\Omega_{\rm o},0}\sigma_0
+\vec{Q}_{\Omega_{\rm o}}\cdot\vec\sigma$ with $Q_{\Omega_{\rm o},0}=0$ and 
$\vec{Q}_{\Omega_{\rm o}}=-\int_{\R^3}(\vec{\Omega}_{\rm o}\times\vec{M})k|k|^2dk$, 
it follows that
\begin{align*}
  J_{W,0}^j &= -\sum_{\ell=1}^3\big(\pa_{x_\ell}Q_0^{j\ell} + (W_0\delta_{j\ell}
	+ \Pi_0^{j\ell})\pa_{x_\ell}V\big) + Q_{\Omega_{\rm o},0}^j \\
	&= -\frac16\pa_{x_j}\int_{\R^3}M_0|k|^4 dk - \frac53W_0\pa_{x_j}V, \\
	\vec{J}_W^j &= -\sum_{\ell=1}^3\big(\pa_{x_\ell}\vec{Q}^{j\ell}
	+ (\vec{W}\delta_{j\ell} + \vec{\Pi}^{j\ell})\pa_{x_\ell}V\big) 
	+ \vec{Q}_{\Omega_{\rm o}}^j \\
	&= -\frac16\pa_{x_j}\int_{\R^3}\vec{M}|k|^4dk - \frac53\vec{W}\pa_{x_j}V
	- \int_{\R^3}(\vec{\Omega}_{\rm o}\times\vec{M})k_j|k|^2 dk.
\end{align*}
Inserting these expressions into \eqref{2.aux1}-\eqref{2.aux4} gives 
\eqref{2.etn0}-\eqref{2.etvecw}.
\end{proof}


\section{Simplified spin energy-transport equations}\label{sec.simpl}

In this section, we derive some explicit models. We assume for simplicity
that the odd part of the magnetization vanishes, $\vec{\Omega}_{\rm o}=0$, and 
that the even part $\vec{\Omega}_{\rm e}$ depends on $x$ only.
Moreover, we suppose that the spin-flip interactions are modeled by the
relaxation-time operator 
\begin{equation}\label{3.sf}
  Q_{\rm sf}(M) := -\frac{1}{\tau_{\rm sf}}
	\bigg(M-\frac12\operatorname{tr}(M)\sigma_0\bigg)
	= -\frac{1}{\tau_{\rm sf}}\vec{M}\cdot\vec\sigma,
\end{equation}
where $\tau_{\rm sf}>0$ is the average time between two subsequent spin-flip
collisions, and we recall that $M=M_0\sigma_0+\vec{M}\cdot\vec\sigma$. 
In particular, with the notation of \eqref{2.Q},
$$
  Q_{{\rm sf},n,0} = 0, \quad \vec{Q}_{{\rm sf},n} = -\frac{\vec{n}}{\tau_{\rm sf}},
	\quad Q_{{\rm sf},W,0} = 0, \quad 
	\vec{Q}_{{\rm sf},W} = -\frac{\vec{W}}{\tau_{\rm sf}}.
$$
Then system \eqref{2.etn0}-\eqref{2.etvecw} reduces to
\begin{align}
  & \pa_t n_0 - \diver\bigg(\frac23\na W_0+n_0\na V\bigg) = 0, \label{3.etn0} \\
	& \pa_t\vec{n} - \sum_{j=1}^3\pa_{x_j}\bigg(\frac23\pa_{x_j}\vec{W}
	+\vec{n}\pa_{x_j}V\bigg) + \vec{\Omega}_{\rm e}\times\vec{n} 
	= -\frac{\vec{n}}{\tau_{\rm sf}}, \label{3.etvecn} \\
	& \pa_t W_0 - \diver\bigg(\frac16\na\int_{\R^3}M_0|k|^2dk + \frac53W_0\na V\bigg)
	- \bigg(\frac23\na W_0+n_0\na V\bigg)\cdot\na V = 0, \label{3.etw0} \\
	& \pa_t\vec{W} - \sum_{j=1}^3\bigg\{\pa_{x_j}\bigg(\frac16\pa_{x_j}\int_{\R^3}
	\vec{M}|k|^4 dk + \frac53\vec{W}\pa_{x_j}V\bigg) + \bigg(\frac23\pa_{x_j}\vec{W}
	+ \vec{n}\pa_{x_j}V\bigg)\pa_{x_j}V\bigg\} \label{3.etvecw} \\
	&\phantom{xxm}{}+ \vec{\Omega}_{\rm e}\times\vec{W} = -\frac{\vec{W}}{\tau_{\rm sf}}.
	\nonumber
\end{align}
Given $(n_0,\vec{n},W_0)$, we define the spin-up/spin-down densities $n_\pm$
and the temperature $T$ by
\begin{equation}\label{3.npm}
  n_\pm = n_0\pm|\vec{n}|, \quad W_0 = \frac32 n_0T.
\end{equation}
We also introduce the Gaussian with standard deviation $\theta>0$,
\begin{equation}\label{3.gT}
  g_\theta(k) = (2\pi\theta)^{-3/2}\exp\bigg(-\frac{|k|^2}{2\theta}\bigg), 
\end{equation}
whose moments are given by
\begin{equation}\label{3.int.gT}
  \int_{\R^3}g_\theta(k)\begin{pmatrix} 1 \\ |k|^2/2 \\ |k|^4/6 \end{pmatrix}dk
	= \begin{pmatrix} 1 \\ 3\theta/2 \\ 5\theta^2/2 \end{pmatrix}.
\end{equation}


\subsection{First model}\label{sec.et1}


\begin{theorem}[Spin energy-transport model with $\vec{c}=0$]\label{thm.czero}
For $\vec{c}=0$, system \eqref{3.etn0}-\eqref{3.etvecw} can be written in
the variables $(n_0,T,\vec{n})$ as \eqref{1.et1.n0}-\eqref{1.et1.vecn}.
\end{theorem}

\begin{proof}
Under the assumption $\vec{c}=0$, the higher-order moments
in \eqref{3.etw0}-\eqref{3.etvecw} can be computed explicitly. Indeed,
the Pauli expansion of the Maxwellian \eqref{2.MPauli} simplifies to
$$
  M_0 = e^{a_0+c_0|k|^2/2}\cosh|\vec{a}|, \quad
	\vec{M} = e^{a_0+c_0|k|^2/2}\sinh|\vec{a}|\frac{\vec{a}}{|\vec{a}|}.
$$
Observe that $c_0<0$ is necessary to ensure the integrability of $M_0$ and
$\vec{M}$. The above expressions can be reformulated by introducing the 
new Lagrange multipliers
$$
  \kappa_\pm := \bigg(\frac{2\pi}{-c_0}\bigg)^{3/2}e^{a_0\pm|\vec{a}|}, \quad
	\theta := -\frac{1}{c_0}, \quad \vec{\gamma} := \frac{\vec{a}}{|\vec{a}|}.
$$
Then $M_0=\frac12(\kappa_+ + \kappa_-)g_\theta(k)$, 
$\vec{M}=\frac12(\kappa_+ - \kappa_-)g_\theta(k)\vec{\gamma}$,
where $g_\theta$ is defined in \eqref{3.gT}.
As a consequence, we have
\begin{align*}
  & n_0 = \int_{\R^3}M_0 dk = \frac12(\kappa_+ + \kappa_-), \quad 
	\vec{n} = \int_{\R^3}\vec{M}dk = \frac12(\kappa_+ - \kappa_-)\vec{\gamma}, \\
  & W_0 = \frac12\int_{\R^3}M_0|k|^2 dk = \frac34\theta(\kappa_+ + \kappa_-),
\end{align*}
and we infer from \eqref{3.npm} that $\kappa_\pm = n_\pm$, 
$\vec{\gamma}=\vec{n}/|\vec{n}|$, and $\theta=T$. 
Then the Pauli coefficients become $M_0=n_0 g_T(k)$, $\vec{M}=\vec{n}g_T(k)$
and
$$
	\vec{W} = \frac12\int_{\R^3}\vec{M}|k|^2 dk = \frac32\vec{n}T, \quad
	\frac16\int_{\R^3}M_0|k|^4 dk = \frac52 n_0T^2.
$$
Inserting these expressions into \eqref{3.etn0}-\eqref{3.etvecw} shows the result.
\end{proof}

\begin{remark}\label{rem.et1}\rm
The derivation of model \eqref{1.et1a.n0}-\eqref{1.et1a.vecn} is similar to
that one in \cite{BeEl09}, therefore we sketch it only. 
The Maxwellian is here given by
\begin{align*}
  & M[F](k) = (2\pi\theta[F])^{-3/2}e^{-|k|^2/(2\theta[F])}\int_{\R^3} F(k')dk', \\
	& \mbox{where }
	\theta[F] = \frac13\frac{\int_{\R^3}\operatorname{tr}(PF(k))|k|^2dk}{\int_{\R^3}
	\operatorname{tr}(PF(k))dk}.
\end{align*}
The formal limit $\eps\to 0$ in \eqref{2.be} gives $Q(F^0)=0$, where
$F^0=\lim_{\eps\to 0}F_\eps$, showing that $F^0=M[F^0]$. Next, we perform
a Hilbert expansion $F_\eps=M[F]+\eps F^1+O(\eps^2)$ and assume that $F^1$
is odd with respect to $k$. Since $1$, $|k|^2/2$ are even functions, $F^1$
does not contribute to the moments $n=\int_{\R^3}Fdk$, 
$W=\frac12\int_{\R^3}F|k|^2dk$. It holds that 
$W=\frac32 nT$, where $T:=\theta[F^0]$.
After a computation which is similar to the
derivation of the semiclassical energy-transport equations, we obtain the
moment equations
\begin{align*}
  & \pa_t n + \diver G_n + \mathrm{i}[n,\vec{\Omega}\cdot\vec{\sigma}] 
	= \frac12\operatorname{tr}(n)-n, \quad 
	G_n = -P^{-1/2}\big(\na(nT) + n\na V\big)P^{-1/2}, \\
	& \frac32\pa_t(nT) + \diver G_W + G_n\cdot\na V = 0, \quad
  G_W = -\frac53 P^{-1/2}\big(\na(nT^2) + nT\na V\big)P^{-1/2}.
\end{align*}
In order to formulate these equations in the Pauli components, 
we observe that for any $2\times 2$
Hermitian matrix $A=a_0\sigma_0+\vec{a}\cdot\vec{\sigma}$, it holds that
$P^{1/2}AP^{1/2}=b_0\sigma_0+\vec{b}\cdot\vec{\sigma}$, where
$$
  \begin{pmatrix} b_0 \\ \vec{b}\end{pmatrix} 
	= \eta^{-2}\begin{pmatrix} 1 & -p\vec{\Omega}^\top \\
	-p\vec{\Omega} & (1-\eta)\vec{\Omega}\otimes\vec{\Omega}+\eta\sigma_0\end{pmatrix}
	\begin{pmatrix} a_0 \\ \vec{a}\end{pmatrix}, \quad \eta=\sqrt{1-p^2}.
$$
We omit the calulcation and only note that
this leads to \eqref{1.et1a.n0}-\eqref{1.et1a.vecn}.
\qed
\end{remark}


\subsection{Second model}\label{sec.et2}


\begin{theorem}[Spin energy-transport model with $\vec{a}=0$, version I]\sloppy
\label{thm.azero}
For $\vec{a}=0$, system \eqref{3.etn0}-\eqref{3.etvecw} can be written in
the variables $(n_0,T,\vec{n})$ as \eqref{1.et2.n0}-\eqref{1.et2.vecn},
where the diffusion coefficient $D(n_+,n_-)$ and the polarization factor
$p(n_+,n_-)$ are defined by
\begin{equation}\label{32.Dp}
  D(n_+,n_-) = \frac{2n_0(n_+^{7/3}+n_-^{7/3})}{(n_+^{5/3}+n_-^{5/3})^2}, \quad
	p(n_+,n_-) = \frac{n_+^{5/3}-n_-^{5/3}}{n_+^{5/3}+n_-^{5/3}},
\end{equation}
and the spin-up/spin-down densities are given by $n_\pm=n_0\pm|\vec{n}|$.
\end{theorem}

\begin{proof}
For $\vec{a}=0$, the Pauli components of the Maxwellian take the form
\begin{equation}\label{32.M}
  M_0 = e^{a_0+c_0|k|^2/2}\cosh\bigg(|\vec{c}|\frac{|k|^2}{2}\bigg), \quad
	\vec{M} = e^{a_0+c_0|k|^2/2}\sinh\bigg(|\vec{c}|\frac{|k|^2}{2}\bigg)
	\frac{\vec{c}}{|\vec{c}|}.
\end{equation}
The integrability of $M_0$ and $\vec{M}$ implies that $c_0\pm|\vec{c}|<0$.
In the new Lagrange multiplier variables
\begin{equation}\label{32.K}
  K := (2\pi)^{3/2}e^{a_0}, \quad \theta_\pm := -\frac{1}{c_0\pm|\vec{c}|}, \quad
	\vec{\gamma} := \frac{\vec{c}}{|\vec{c}|}
\end{equation}
these components can be rewritten as
$$
  M_0 = \frac{K}{2}\big(\theta_+^{3/2}g_{\theta_+}(k) + \theta_-^{3/2}g_{\theta_-}(k)
	\big), \quad
	\vec{M} = \frac{K}{2}\big(\theta_+^{3/2}g_{\theta_+}(k) 
	- \theta_-^{3/2}g_{\theta_-}(k)\big)\vec{\gamma}.
$$
Taking into account \eqref{3.int.gT}, this shows that
\begin{align*}
  & n_0 = \int_{\R^3}M_0 dk = \frac{K}{2}\big(\theta_+^{3/2} + \theta_-^{3/2}\big), 
	\quad \vec{n} = \int_{\R^3}\vec{M}dk 
	= \frac{K}{2}\big(\theta_+^{3/2} - \theta_-^{3/2}\big)\vec{\gamma}, \\
	& W_0 = \frac12\int_{\R^3}M_0|k|^2 dk 
	= \frac{3K}{4}\big(\theta_+^{5/2} + \theta_-^{5/2}\big),
\end{align*}
and consequently, $n_\pm:=n_0\pm|\vec{n}|=K\theta_\pm^{3/2}$, 
$\vec{\gamma}=\vec{n}/|\vec{n}|$.
This implies that $W_0=\frac34(n_+\theta_+ + n_-\theta_-)$ and
$n_-/n_+ = (\theta_-/\theta_+)^{3/2}$. Hence,
\begin{align*}
  \frac32 n_0T &= W_0 = \frac34\frac{\theta_+}{n_+^{2/3}}\bigg(n_+^{5/3}
	+ \bigg(\frac{n_+}{n_-}\bigg)\frac{\theta_+}{\theta_-}n_-\bigg) \\
	&= \frac{3\theta_+}{4n_+^{2/3}}\big(n_+^{5/3} + n_-^{5/3}\big)
	= \frac{3\theta_-}{4n_-^{2/3}}\big(n_+^{5/3} + n_-^{5/3}\big).
\end{align*}
We obtain the following form for the Pauli components of $M$:
$$
  M_0 = \frac12\bigg(n_+g_{\theta_+}(k)+n_-g_{\theta_-}(k)\bigg), \quad
	\vec{M} = \frac12\bigg(n_+g_{\theta_+}(k)+n_-g_{\theta_-}(k)\bigg)
	\frac{\vec{n}}{|\vec{n}|}.
$$
It remains to compute the higher-order moments:
\begin{align*}
  \vec{W} = \frac12\int_{\R^3}\vec{M}|k|^2 dk
	&= \frac34(n_+\theta_+ - n_-\theta_-)\frac{\vec{n}}{|\vec{n}|}
	= \frac32n_0T\frac{n_+^{5/3} - n_-^{5/3}}{n_+^{5/3} + n_-^{5/3}}
	\frac{\vec{n}}{|\vec{n}|}, \\
	\frac16\int_{\R^3}M_0|k|^4 dk &= \frac54(n_+\theta_+^2 + n_-\theta_-^2)
	= 5n_0^2T^2\frac{n_+^{7/3} + n_-^{7/3}}{(n_+^{5/3} + n_-^{5/3})^2}.
\end{align*}
Inserting these expressions into \eqref{3.etn0}-\eqref{3.etw0} concludes the proof.
\end{proof}

\begin{remark}\rm\label{rem.D}
Equations \eqref{1.et2.n0}-\eqref{1.et2.vecn} are fully coupled since the diffusion
coefficient $D(n_+,n_-)$ depends on the spin vector density through
$|\vec{n}| = (n_+ - n_-)/2$. However, it turns out that $1\le D(n_+,n_-)\le 1.1$
for $|\vec{n}|\le n_0$, which means that the dependence of the energy $\frac32 n_0T$
on the spin vector density $\vec{n}$ is in fact very weak.
When the spin vector density vanishes, $\vec{n}=0$, it follows that $n_+=n_-=n_0$
and $D(n_+,n_-)=1$, and we recover the classical energy-transport model.
\qed
\end{remark}

The model in Theorem \ref{thm.czero} can be equivalently formulated in the variables
$(n_0,W_0,\vec{W})$, and this formulation is used below in the existence analysis.

\begin{theorem}[Spin energy-transport model with $\vec{a}=0$, version II]
\label{thm.azero2}\sloppy
For $\vec{a}=0$, system \eqref{3.etn0}-\eqref{3.etvecw} can be written in
the variables $(n_0,T,\vec{W})$ as
\begin{align}
  & \pa_t n_0 - \diver\big(\na(n_0T)+n_0\na V\big) = 0, \label{3.et2.n02} \\
	& \frac32\pa_t(n_0T) - \diver\bigg(\na Z_0+\frac52 n_0T\na V\bigg)
	- \big(\na(n_0T)+n_0\na V\big)\cdot\na V = 0, \label{3.et2.n0T2} \\
  & \pa_t\vec{W} - \sum_{j=1}^3\pa_{x_j}\bigg(\pa_{x_j}\vec{Z}+\frac53\vec{W}\pa_{x_j}V
	\bigg) - \sum_{j=1}^3\bigg(\frac23\pa_{x_j}\vec{W}
	+ \vec{n}\na_{x_j}V\bigg)\pa_{x_j}V	\label{3.et2.vecw2} \\
	&\phantom{xx}{}+ \vec{\Omega}_{\rm e}\times\vec{W} = -\frac{\vec{W}}{\tau_{\rm sf}},
  \nonumber
\end{align}
where the spin-vector density $\vec{n}$ and the auxiliary quantities 
$Z_0$ and $\vec{Z}$ are given by
\begin{align} 
  \vec{n} &= n_0\frac{W_+^{3/5}-W_-^{3/5}}{W_+^{3/5}+W_-^{3/5}}
	\frac{\vec{W}}{|\vec{W}|}, \nonumber \\
	Z_0 &= \frac{5}{18n_0}\big(W_+^{3/5}+W_-^{3/5}\big)\big(W_+^{7/5}+W_-^{7/5}\big), 
	\label{32.Z0} \\
  \vec{Z} &= \frac{5}{18n_0}\big(W_+^{3/5}+W_-^{3/5}\big)
	\big(W_+^{7/5}-W_-^{7/5}\big)\frac{\vec{W}}{|\vec{W}|}, \label{32.vecZ}
\end{align}
and $W_\pm=W_0\pm|\vec{W}|$, $W_0=\frac32n_0T$.
\end{theorem}

\begin{proof}
With the new Lagrange multipliers introduced in the proof of Theorem
\ref{thm.czero}, we find that
\begin{align}
  n_0 &= \int_{\R^3}M_0 dk = \frac{K}{2}\big(\theta_+^{3/2}+\theta_-^{3/2}\big),
	\quad W_0 = \frac12\int_{\R^3}M_0|k|^2 dk 
	= \frac{3K}{4}\big(\theta_+^{5/2}+\theta_+^{5/2}\big), \label{32.n0} \\
	\vec{W} &= \frac12\int_{\R^3}\vec{M}|k|^2 dk 
	= \frac{3K}{4}\big(\theta_+^{5/2}-\theta_-^{5/2}\big)\vec{\gamma}. \label{32.vecw}
\end{align}
As $c_0<0$ is required to ensure integrability of the Maxwellian, it holds that
$\theta_+\ge\theta_-$, such that we deduce from \eqref{32.vecw}
that 
\begin{equation}\label{32.gamma}
  \vec{\gamma} = \frac{\vec{W}}{|\vec{W}|}, \quad 
	|\vec{W}| = \frac{3K}{4}(\theta_+^{5/2}-\theta_-^{5/2}).
\end{equation}
Let $W_\pm=W_0\pm|\vec{W}|$. Then
$$
  W_\pm = \frac{3K}{4}\big(\theta_+^{5/2}+\theta_+^{5/2}\big)
	\pm \frac{3K}{4}\big(\theta_+^{5/2}-\theta_-^{5/2}\big)
	= \frac{3K}{2}\theta_\pm^{5/2},
$$
which is equivalent to $\theta_\pm = (2W_\pm/(3K))^{2/5}$. Inserting this
expression into the first equation of \eqref{32.n0}, we obtain
$$
  n_0 = \frac{K}{2}\bigg(\bigg(\frac{2W_+}{3K}\bigg)^{3/5}
	+ \bigg(\frac{2W_-}{3K}\bigg)^{3/5}\bigg) 
	= \frac{K^{2/5}}{2^{2/5}3^{3/5}}\big(W_+^{3/5} + W_-^{3/5}\big).
$$
Thus, the constant $K$ can be written as 
\begin{equation}\label{32.K2}
  K = 2\cdot 3^{3/2}n_0^{5/2}(W_+^{3/5}+W_-^{3/5})^{-5/2}, 
\end{equation}
and we can eliminate $K$ in the formulation of $\theta_\pm$:
\begin{equation}\label{32.thetapm}
  \theta_\pm = \bigg(\frac{2W_\pm}{3K}\bigg)^{2/5}
	= \frac{W_\pm^{2/5}}{3n_0}\big(W_+^{3/5} + W_-^{3/5}\big).
\end{equation}
The spin-vector density is then computed as follows:
\begin{align*}
  \vec{n} &= \int_{\R^3}\vec{M}dk 
	= \frac{K}{2}\big(\theta_+^{3/2}-\theta_-^{3/2}\big)\frac{\vec{W}}{|\vec{W}|} \\
	&= \frac{3^{3/2}n_0^{5/2}}{(W_+^{3/5}+W_-^{3/5})^{5/2}}
	\bigg(\frac{W_+^{3/5}+W_-^{3/5}}{3n_0}\bigg)^{3/2}
	\big(W_+^{3/5}-W_-^{3/5}\big)\frac{\vec{W}}{|\vec{W}|} \\
	&= n_0\frac{W_+^{3/5}-W_-^{3/5}}{W_+^{3/5}+W_-^{3/5}}\frac{\vec{W}}{|\vec{W}|}.
\end{align*}

It remains to compute the fourth-order moments. 
Using $W_\pm=\frac32 K\theta_\pm^{5/2}$
and \eqref{32.thetapm}, we have
\begin{align*}
  \frac16\int_{\R^3}M_0|k|^4 dk &= \frac{5K}{4}\big(\theta_+^{7/2}+\theta_-^{7/2}\big)
	= \frac56(\theta_+W_+ + \theta_- W_-) \\
	&= \frac{5}{18n_0}\big(W_+^{3/5}+W_-^{3/5}\big)\big(W_+^{7/5}+W_-^{7/5}\big).
\end{align*}
In an analogous way, we calculate
$$
  \frac16\int_{\R^3}\vec{M}|k|^4 dk 
	= \frac{5}{18n_0}\big(W_+^{3/5}+W_-^{3/5}\big)\big(W_+^{7/5}-W_-^{7/5}\big)
	\frac{\vec{W}}{|\vec{W}|}.
$$
Inserting these expressions into \eqref{3.etn0}-\eqref{3.etvecw}, the result
follows.
\end{proof}

\begin{remark}\rm
If $\vec{W}=0$, it follows that $W_\pm=W_0=\frac32n_0T$, $\vec{n}=0$, 
$Z_0=\frac52n_0T^2$, and we recover the semiclassical energy-transport model.
It is possible to see that $1\le Z_0/(\frac52n_0T^2)\le 1.08$, which shows that
the coupling is rather weak. This is expected since the coupling 
in system \eqref{1.et2.n0}-\eqref{1.et2.vecn} is weak too.
\qed
\end{remark}


\subsection{Third model}\label{sec.et3}

\begin{theorem}[Spin energy-transport model for $\vec{a}=\lambda\vec{c}$]
\label{thm.ac}
Under the assumption $\vec{a}=\lambda\vec{c}$ for some $\lambda=\lambda(x,t)\ge 0$,
system \eqref{3.etn0}-\eqref{3.etvecw} can be written in the variables
$(n_\pm,W_\pm,\vec{s})$ as \eqref{1.et3.npm}-\eqref{1.et3.J},
where $(n_\pm,T_\pm,\vec{s})$ are linked to $(n_0,\vec{n},W_0,\vec{W})$ via
\begin{equation}\label{3.nmp}
  n_\pm = n_0\pm|\vec{n}|, \quad \frac32 n_\pm T_\pm = W_0\pm|\vec{W}|, \quad
  \vec{s}=\frac{\vec{n}}{|\vec{n}|}.
\end{equation}
\end{theorem}

\begin{proof}
First, we compute the moments in order to make system \eqref{3.etn0}-\eqref{3.etvecw}
explicit. Under the assumption that $\vec{a}=\lambda\vec{c}$ for some 
$\lambda\ge 0$, the Pauli components of the Maxwellian become
\begin{align*}
  M_0 &= \frac12\exp\bigg(a_0 + \lambda|\vec{c}| + \frac12(c_0+|\vec{c}|)|k|^2\bigg)
	+ \frac12\exp\bigg(a_0 - \lambda|\vec{c}| + \frac12(c_0-|\vec{c}|)|k|^2\bigg), \\
	\vec{M} &= \frac12\bigg\{\exp\bigg(a_0 + \lambda|\vec{c}| 
	+ \frac12(c_0+|\vec{c}|)|k|^2\bigg) - \exp\bigg(a_0 - \lambda|\vec{c}|
	+ \frac12(c_0-|\vec{c}|)|k|^2\bigg)\bigg\}\frac{\vec{c}}{|\vec{c}|}.
\end{align*}
Introducing the new Lagrange multipliers
$$
  \kappa_\pm := (2\pi\theta_\pm)^{3/2}e^{a_0\pm\lambda|\vec{c}|}, \quad
	\theta_\pm := -\frac{1}{c_0\pm|\vec{c}|}, \quad 
	\vec{\gamma} := \frac{\vec{c}}{|\vec{c}|},
$$
the Pauli components of $M$ can be rewritten as
$$
  M_0 = \frac12(k_+g_{\theta_+} + k_- g_{\theta_-}), \quad
	\vec{M} = \frac12(k_+ g_{\theta_+} - k_- g_{\theta_-})\vec{\gamma},
$$
where $g_{\theta_\pm}$ is defined in \eqref{3.gT}. Since the Maxwellian
has to be integrable, we have $c_0+|\vec{c}|<0$ and consequently,
$\theta_+\ge \theta_->0$ and $\kappa_+\ge \kappa_->0$. It follows that
\begin{align*}
  n_0 &= \frac12(\kappa_+ + \kappa_-), & \vec{n} 
	&= \frac12(\kappa_+ - \kappa_-)\vec{\gamma}, \\
	W_0 &= \frac34(\kappa_+\theta_+ + \kappa_-\theta_-), &
	\vec{W} &= \frac34(\kappa_+\theta_+ - \kappa_-\theta_-)\vec{\gamma}.
\end{align*}
These expressions allow us to identify the new Lagrange multipliers with
$n_\pm=n_0\pm|\vec{n}|$, $W_\pm=W_0\pm|\vec{W}|$, and $\vec{s}=\vec{n}/|\vec{n}|$:
$$
  n_\pm = k_\pm, \quad W_\pm = \frac32 k_\pm\theta_\pm, \quad
	\vec{s} = \frac{\vec{n}}{|\vec{n}|} = \frac{\vec{W}}{|\vec{W}|} = \vec{\gamma}.
$$
The last expression represents a constraint of the spin part of the particle
density and energy. Moreover, the definition $T_\pm = 2W_\pm/(3n_\pm)$ implies that
$T_\pm = \theta_\pm$. Thus, the Pauli components of the Maxwellian take the form
\begin{equation}\label{3.max3}
  M_0 = \frac12(n_+ g_{T_+} + n_- g_{T_-}), \quad
	\vec{M} = \frac12(n_+ g_{T_+} - n_- g_{T_-})\vec{s}.
\end{equation}

Computing the higher-order moments
\begin{align*}
  \frac16\int_{\R^3}M_0|k|^4 dk &= \frac{1}{12}\int_{\R^3}
	(n_+ g_{\theta_+} + n_- g_{\theta_-})|k|^4 dk = \frac54(n_+ T_+^2 + n_- T_-^2), \\
	 \frac16\int_{\R^3}\vec{M}|k|^4 dk &= \frac54(n_+ T_+^2 - n_- T_-^2)\vec{s},
\end{align*}
system \eqref{3.etn0}-\eqref{3.etvecw} becomes
\begin{align}
  & \pa_t n_0 - \diver\bigg(\frac23\na W_0 + n_0\na V\bigg) = 0, \label{3.aux1} \\
	& \pa_t\vec{n} - \diver\bigg(\frac23\na\vec{W} + \vec{n}\na V\bigg)
	+ \vec{\Omega}_{\rm e}\times\vec{n} = -\frac{\vec{n}}{\tau_{\rm sf}}, 
	\label{3.aux2} \\
	& \pa_t W_0 - \diver\bigg(\frac59\na\bigg(\frac{W_+^2}{n_+} + \frac{W_-^2}{n_-}
	\bigg) + \frac53 W_0\na V\bigg) - \bigg(\frac23\na W_0 + n_0\na V\bigg)\cdot\na V 
	= 0, \label{3.aux3} \\
	& \pa_t\vec{W} - \diver\bigg(\frac59\na
	\bigg(\bigg(\frac{W_+^2}{n_+} + \frac{W_-^2}{n_-}\bigg)\vec{s}\bigg) 
	+ \frac53 \vec{W}\na V\bigg) - \bigg(\frac23\na\vec{W} + \vec{n}\na V\bigg)
	\cdot\na V \label{3.aux4} \\
	&\phantom{xxm}{}+ \vec{\Omega}_{\rm e}\times\vec{W} 
	= -\frac{\vec{W}}{\tau_{\rm sf}}. \nonumber
\end{align}

The next step is to reformulate this system in terms of $(n_\pm,W_\pm,\vec{s})$.
First, we derive \eqref{1.et3.npm}. For this, we take the scalar product of 
\eqref{3.aux2} and 
$\vec{s}=\vec{n}/|\vec{n}|=\vec{W}/|\vec{W}|$, leading to 
\begin{equation}\label{3.aux5}
  \pa_t|\vec{n}| - \diver\bigg(\frac23\na|\vec{W}| + |\vec{n}|\na V\bigg)
	+ \na\vec{s}\cdot\bigg(\frac23\na\vec{W} + \vec{n}\na V\bigg) 
	= -\frac{|\vec{n}|}{\tau_{\rm sf}}.
\end{equation}
Observing that $|\vec{s}|=1$ implies that $\na\vec{s}\cdot\vec{s}=0$, we find that
$\na\vec{s}\cdot\na\vec{W} = \na\vec{s}\cdot\na(|\vec{W}|\vec{s}) = |\vec{W}|
|\na\vec{s}|^2$ and $\na\vec{s}\cdot\vec{n}=0$. Hence, \eqref{3.aux5} becomes
$$
  \pa_t|\vec{n}| - \diver\bigg(\frac23\na|\vec{W}| + |\vec{n}|\na V\bigg)
	+ \frac23|\vec{W}||\na\vec{s}|^2 = -\frac{|\vec{n}|}{\tau_{\rm sf}}.
$$
Taking the sum and difference of equation \eqref{3.etn0} for $n_0$ 
and the previous equation,
we obtain \eqref{1.et3.npm} using $n_\pm=n_0\pm|\vec{n}|$ and
$|\vec{W}|=\frac34(n_+T_+ - n_-T_-)$.

Second, we derive \eqref{1.et3.nT}. Multiplying \eqref{3.aux4} by
$\vec{s}=\vec{W}/|\vec{W}|$ yields
\begin{align*}
  \pa_t|\vec{W}| &- \diver\bigg(\frac59\na\bigg(\frac{W_+^2}{n_+}-\frac{W_-^2}{n_-}
	\bigg) + \frac53|\vec{W}|\na V\bigg) \\
	&{}+ \na\vec{s}\cdot\bigg(\frac59\na\bigg(\bigg(\frac{W_+^2}{n_+}-\frac{W_-^2}{n_-}
	\bigg)\vec{s}\bigg) + \frac53|\vec{W}|\na V\bigg) \\
	&{}- \vec{s}\cdot\bigg(\frac23(\na V\cdot\na)\vec{W} + \vec{n}|\na V|^2\bigg)
	= -\frac{|\vec{W}|}{\tau_{\rm sf}},
\end{align*}
and adding and subtracting this equation from \eqref{3.aux3} and employing
$\vec{s}\cdot\vec{n}=|\vec{n}|$ shows that
$\frac32 n_\pm T_\pm=W_\pm=W_0\pm|\vec{W}|$ solves \eqref{1.et3.nT}.

Third, we derive \eqref{1.et3.s}. We take the product of 
$P:=({\mathbb I}-\vec{s}\otimes\vec{s})/|\vec{n}|$ and \eqref{3.aux2}
(here, $\mathbb I$ is the unit matrix in $\R^{3\times 3}$.) 
This matrix has the following properties: $P\vec{n}=0$,
$P\pa_t\vec{n}=\pa_t\vec{s}$, and $P\na\vec{n}=\na\vec{s}$.
A computation shows that
\begin{equation}\label{3.aux6}
  \pa_t\vec{s} - \diver\bigg(P\bigg(\frac23\na\vec{W}+\vec{n}\na V\bigg)\bigg)
	+ \na P\cdot\bigg(\frac23\na\vec{W}+\vec{n}\na V\bigg) + \vec{\Omega}_{\rm e}
	\times\vec{s} = 0.
\end{equation}
We reformulate the second and third term:
\begin{align*}
  P\bigg(\frac23\na\vec{W}+\vec{n}\na V\bigg)
	&= \frac23\frac{|\vec{W}|}{|\vec{n}|}\frac{1}{|\vec{W}|}
	({\mathbb I}-\vec{s}\otimes\vec{s})\na\vec{W} 
	=  \frac23\frac{|\vec{W}|}{|\vec{n}|}\na\vec{s}, \\
	(\na P\cdot\na\vec{W})_i &= \sum_{j=1}^3\na P_{ij}\cdot\na W_j \\
	&= \sum_{j=1}^3\bigg(\frac{1}{|\vec{n}|^2}(\delta_{ij}-s_is_j)\na|\vec{n}|
	- \frac{1}{|\vec{n}|}(s_i\na s_j+s_j\na s_i)\bigg)\cdot\na W_j \\
	&= -\frac{1}{|\vec{n}|^2}|\vec{W}|\na|\vec{n}|\cdot\na s_i
	- \frac{1}{|\vec{n}|}|\vec{W}||\na \vec{s}|^2 s_i 
	- \frac{1}{|\vec{n}|}\na|\vec{W}|\na s_i \\
	&= -\frac{|\vec{W}|}{|\vec{n}|}\na\log(|\vec{n}||\vec{W}|)\cdot\na s_i
	- \frac{|\vec{W}|}{|\vec{n}|}|\na\vec{s}|^2 s_i, \\
	\na P\cdot\vec{n}\na V &= (\na V\cdot\na)(P\vec{n})
	- P(\na V\cdot\na)\vec{n} = -\na V\cdot\na\vec{s}.
\end{align*}
Therefore, \eqref{3.aux6} becomes
\begin{align*}
  \pa_t\vec{s} &= \diver\bigg(\frac23\frac{|\vec{W}|}{|\vec{n}|}\na\vec{s}\bigg)
	+ \bigg(\frac23\frac{|\vec{W}|}{|\vec{n}|}\na\log(|\vec{n}||\vec{W}|)
	+ \na V\bigg)\cdot\na\vec{s} + \frac23\frac{|\vec{W}|}{|\vec{n}|}|\na\vec{s}|^2
	\vec{s} - \vec{\Omega}_{\rm e}\times\vec{s} \\
	&= \frac23\frac{|\vec{W}|}{|\vec{n}|}(\Delta\vec{s}+|\na\vec{s}|^2\vec{s})
	+ \bigg\{\na\bigg(\frac23\frac{|\vec{W}|}{|\vec{n}|}\bigg)
	+ \frac23\frac{|\vec{W}|}{|\vec{n}|}\na\log(|\vec{n}||\vec{W}|) + \na V\bigg\}
	\cdot\na\vec{s} - \vec{\Omega}_{\rm e}\times\vec{s} \\
	&= \frac23\frac{|\vec{W}|}{|\vec{n}|}\vec{s}\times(\Delta\vec{s}\times\vec{s})
	+ \bigg(\frac43\frac{\na|\vec{W}|}{|\vec{n}|} + \na V\bigg)\cdot\na\vec{s}
	- \vec{\Omega}_{\rm e}\times\vec{s} \\
	&= \frac23\frac{W_+ - W_-}{n_+ - n_-}\vec{s}\times(\Delta\vec{s}\times\vec{s})
	+ \bigg(\frac43\frac{\na(W_+ - W_-)}{n_+ - n_-} + \na V\bigg)\cdot\na\vec{s}
	- \vec{\Omega}_{\rm e}\times\vec{s}.
\end{align*}
Then, using $W_\pm=\frac32n_\pm T_\pm$, equation \eqref{1.et3.s} follows.
\end{proof}

\begin{remark}\label{rem.LLG}\rm
If the temperature is constant, $T_+=T_-=1$, equation \eqref{1.et3.s} for
the spin accumulation vector becomes
$$
  \pa_t\vec{s} - \vec{s}\times(\Delta\vec{s}\times\vec{s})
	= \na(\log|\vec{n}|^2 + V)\cdot\na\vec{s} - \vec{\Omega}_{\rm e}\times\vec{s}.
$$
If $\vec{\Omega}_{\rm e}=\Delta\vec{s}$, this resembles the Landau-Lifshitz
equation with the exception of the first term on the right-hand side,
which provides an additional field contribution. Note that this term does not
vanish in termal equilibrium where $V=-\log n_0$.
\qed
\end{remark}


\section{Entropy structure}\label{sec.ent}

In this section, we investigate the entropy structure of 
the spin energy-transport equations derived in the previous section.
Recall that the entropy of the general model is given by
$$
  H = \int_{\R^3}\int_{\R^3}\operatorname{tr}(M\log M)dkdx,
$$
where $M=M_0\sigma_0 + \vec{M}\cdot\vec{\sigma}$ is the Maxwellian.
We introduce $M_\pm=M_0\pm|\vec{M}|$ and 
$P_\pm=\frac12(\sigma_0\pm(\vec{M}/|\vec{M}|)\cdot\vec{\sigma})$.
Then $(P_+,P_-)$ is a set of complete orthogonal projections since
$P_\pm^2=P_\pm$, $P_+P_-=0$, and $P_+ + P_-=\sigma_0$. Therefore, for any
function $f:\R\to\R$, 
\begin{align*}
  f(M) &= f(M_+)P_+ + f(M_-)P_- \\
	&= \frac12\big(f(M_+)+f(M_-)\big) \sigma_0
	+ \frac12\big(f(M_+)-f(M_-)\big)\frac{\vec{M}}{|\vec{M}|}\cdot\vec{\sigma}.
\end{align*}
In particular, since the Pauli matrices are traceless,
\begin{equation}\label{4.H}
  H = \frac12\int_{\R^3}\int_{\R^3}(M_+\log M_+ + M_-\log M_-)dkdx.
\end{equation}

\subsection{Entropy inequality for the first model}\label{sec.ent1}

We wish to explore the entropy structure of the first model 
\eqref{1.et1.n0}-\eqref{1.et1.vecn} ($\vec{c}=0$), neglecting the electric field:
\begin{equation}
  \pa_t n_0 = \Delta(n_0T), \quad \frac32\pa_t(n_0T) = \frac52\Delta(n_0T^2), \quad
	\pa_t\vec{n} = \Delta(\vec{n}T) - \vec{\Omega}_{\rm e}\times\vec{n}
  - \frac{\vec{n}}{\tau_{\rm sf}}, \label{4.m1}
\end{equation}
where $x\in\R^3$, $t>0$.
We claim that the entropy is given by \eqref{1.H1}. 
Indeed, since $\vec{c}=0$ by assumption,
$M=g_T(k)(n_0\sigma_0 + \vec{n}\cdot\vec{\sigma})$, where $g_T(k)$ is defined
in \eqref{3.gT} (see the proof of Theorem \ref{thm.czero}). 
Then $M_\pm = g_T(k)n_\pm$ and \eqref{4.H} shows that
\begin{align*}
  H_1 &= \frac12\int_{\R^3}\int_{\R^3}g_T(k)\big(n_+\log(n_+ g_T(k))
	+ n_- \log(n_- g_T(k))\big)dxdk \\
	&= \frac12\int_{\R^3}(n_+\log n_+ + n_-\log n_-)dx
	+ \int_{\R^3}(n_+ + n_-)\int_{\R^3}g_T(k)\log g_T(k)dk dx \\
	&= \frac12\int_{\R^3}(n_+\log n_+ + n_-\log n_-)dx
	- \int_{\R^3}(n_+ + n_-)\bigg(\frac32 + \frac32\log(2\pi T)\bigg)dx.
\end{align*}
Thus, since $\int_{\R^3}(n_+ + n_-)dx$ is constant in time, we find that,
up to a constant,
$$
  H_1 = \int_{\R^3}\big(n_+\log(n_+ T^{-3/2})+n_-\log(n_- T^{-3/2})\big)dx,
$$
which is exactly \eqref{1.H1}. Recall that $n_\pm=n_0\pm|\vec{n}|$.

\begin{proposition}[Entropy inequality for system \eqref{4.m1}]\label{prop.ent1}
The entropy \eqref{1.H1}, considered as a function of time, is nonincreasing 
along (smooth) solutions $(n_0,T,\vec{n})$ to \eqref{4.m1}, and
\begin{align}\label{41.ent1}
  \frac{dH_1}{dt} &+ 4\int_{\R^3}\big(|\na\sqrt{n_+ T}|^2 + |\na\sqrt{n_- T}|^2\big)dx
	+ 20\int_{\R^3}n_0|\na\sqrt{T}|^2\big)dx \\
	&{}+ \frac12\int_{\R^3}(n_+ - n_-)
  (\log n_+ - \log n_-)\bigg(\frac{1}{\tau_{\rm sf}} 
	+ T\bigg|\na\frac{\vec{n}}{|\vec{n}|}\bigg|^2\bigg)dx = 0. \nonumber
\end{align}
\end{proposition}

\begin{proof}
We compute
\begin{align}
  \frac{dH_1}{dt} &= \int_{\R^3}\sum_{s=\pm}\bigg(\log(n_s T^{-3/2})\pa_t n_s
	- \frac32\frac{1}{T}\pa_t(n_s T)\bigg)dx \nonumber \\
	&= \int_{\R^3}\bigg(\log\big((n_0+|\vec{n}|)T^{-3/2}\big)\pa_t(n_0+|\vec{n}|)
	- \frac32\frac{1}{T}\pa_t(n_0T+|\vec{n}|T) \nonumber \\
	&\phantom{xx}{}+ \log\big((n_0-|\vec{n}|)T^{-3/2}\big)\pa_t(n_0-|\vec{n}|)
	- \frac32\frac{1}{T}\pa_t(n_0T-|\vec{n}|T)\bigg)dx \nonumber \\
	&= \int_{\R^3}\bigg\{\log\bigg(\frac{n_0^2-|\vec{n}|^2}{T^3}\bigg)\pa_t n_0
	+ \log\bigg(\frac{n_0+|\vec{n}|}{n_0-|\vec{n}|}\bigg)\frac{\vec{n}}{|\vec{n}|}\cdot
	\pa_t\vec{n} - \frac{2}{T}\pa_t\bigg(\frac32 n_0T\bigg)\bigg\}dx. \label{41.aux}
\end{align}
Inserting \eqref{1.et1.n0} in the first term and integrating by parts, we find that
$$
  \int_{\R^3}\log\bigg(\frac{n_0^2-|\vec{n}|^2}{T^3}\bigg)\pa_t n_0 dx
	= -\int_{\R^3}\na\log\bigg(\frac{n_0^2-|\vec{n}|^2}{T^3}\bigg)\cdot\na(n_0T)dx.
$$
Furthermore, using \eqref{1.et1.vecn} in the second term on the right-hand side
of \eqref{41.aux} and integrating by parts gives
\begin{align*}
  \int_{\R^3} & \log\bigg(\frac{n_0+|\vec{n}|}{n_0-|\vec{n}|}\bigg)
	\frac{\vec{n}}{|\vec{n}|}\cdot\pa_t\vec{n}dx
	= -\int_{\R^3}\na\bigg(\log\bigg(\frac{n_0+|\vec{n}|}{n_0-|\vec{n}|}\bigg)
	\bigg)\frac{\vec{n}}{|\vec{n}|}\cdot\na(\vec{n}T)dx \\
	&{}- \int_{\R^3}\log\bigg(\frac{n_0+|\vec{n}|}{n_0-|\vec{n}|}\bigg)
	\na\frac{\vec{n}}{|\vec{n}|}\cdot\na(\vec{n}T)dx
	- \frac{1}{\tau_{\rm sf}}\int_{\R^3}|\vec{n}|
	\log\bigg(\frac{n_0+|\vec{n}|}{n_0-|\vec{n}|}\bigg)dx.
\end{align*}
Since $\na\vec{n}\cdot\vec{n}=0$ and 
\begin{align*}
  \na\frac{\vec{n}}{|\vec{n}|}\cdot\na(\vec{n}T)
	&= \na\frac{\vec{n}}{|\vec{n}|}\cdot\na\bigg(|\vec{n}|T\frac{\vec{n}}{|\vec{n}|}
	\bigg) = |\vec{n}|T\bigg|\na\frac{\vec{n}}{|\vec{n}|}\bigg|^2, \\
	\frac{\vec{n}}{|\vec{n}|}\cdot\na(\vec{n}T)
	&= \frac{\vec{n}}{|\vec{n}|}\cdot(T\vec{n} + \vec{n}\na T)
	= T\na|\vec{n}| + |\vec{n}|\na T = \na(|\vec{n}|T),
\end{align*}
it follows that
\begin{align*}
  \int_{\R^3} & \log\bigg(\frac{n_0+|\vec{n}|}{n_0-|\vec{n}|}\bigg)
	\frac{\vec{n}}{|\vec{n}|}\cdot\pa_t\vec{n}dx
	= -\int_{\R^3}\log\bigg(\frac{n_0+|\vec{n}|}{n_0-|\vec{n}|}\bigg)
	\cdot\na(|\vec{n}|T)dx \\
	&{}- \int_{\R^3}\log\bigg(\frac{n_0+|\vec{n}|}{n_0-|\vec{n}|}\bigg)
	|\vec{n}|T\bigg|\na\frac{\vec{n}}{|\vec{n}|}\bigg|^2dx
	- \frac{1}{\tau_{\rm sf}}\int_{\R^3}|\vec{n}|
	\log\bigg(\frac{n_0+|\vec{n}|}{n_0-|\vec{n}|}\bigg)dx.
\end{align*}
Finally, we employ \eqref{1.et1.n0T} to reformulate the last term on the
right-hand side of \eqref{41.aux}:
\begin{align*}
  -\int_{\R^3}\frac{2}{T}\pa_t\bigg(\frac32 n_0T\bigg)dx
	&= 5\int_{\R^3}\na\frac{1}{T}\cdot\na(n_0T^2)dx \\
	&= -5\int_{\R^3}\na\log T\cdot\na(n_0T)dx
	- 5\int_{\R^3}\frac{n_0}{T}|\na T|^2 dx.
\end{align*}

Summarizing these expressions, we have
\begin{align*}
  \frac{dH_1}{dt}
	&= -\int_{\R^3}\bigg\{
	\na\log\bigg(\frac{n_0^2-|\vec{n}|^2}{T^3}\bigg)\cdot\na(n_0T)dx
	+\na\log\bigg(\frac{n_0+|\vec{n}|}{n_0-|\vec{n}|}\bigg)\cdot
	\na(|\vec{n}|T) \\
	&\phantom{xx}{}+ 5\na\log T\cdot\na(n_0T) + 5\frac{n_0}{T}|\na T|^2\bigg\}dx \\
  &\phantom{xx}{}- \int_{\R^3}\bigg\{
	\log\bigg(\frac{n_0+|\vec{n}|}{n_0-|\vec{n}|}\bigg)
	|\vec{n}|T\bigg|\na\frac{\vec{n}}{|\vec{n}|}\bigg|^2
	+ \frac{1}{\tau_{\rm sf}}|\vec{n}|
	\log\bigg(\frac{n_0+|\vec{n}|}{n_0-|\vec{n}|}\bigg)
	\bigg\}dx \\
	&= I_1 + I_2.
\end{align*}
The integrals in $I_2$ correspond, up to the minus sign, the second and third
integrals in \eqref{41.ent1}. It remains to show that $I_1$ corresponds to
the first integral in \eqref{41.ent1}, up to the sign. Indeed, since
$\log(n_0^2-|\vec{n}|^2)=\log n_+ + \log n_-$ and
$\log((n_0+|\vec{n}|)/(n_0-|\vec{n}|))=\log n_+ - \log n_-$, we have
\begin{align*}
  I_1 &= -\int_{\R^3}\na\log(n_0^2-|\vec{n}|^2)\cdot\na(n_0T)dx \\
	&\phantom{xx}{}- \int_{\R^3}\na\log\bigg(\frac{n_0+|\vec{n}|}{n_0-|\vec{n}|}\bigg)
	\cdot\na(|\vec{n}|T)dx - 2\int_{\R^3}\na\log T\cdot\na(n_0T)dx \\
  &= -\int_{\R^3}\big(\na\log n_+\cdot\na(n_+T) + \na\log n_-\cdot\na(n_-T)
	+ \na\log T\cdot\na(n_+T + n_-T)\big)dx \\
	&= -\int_{\R^3}\big(\na\log(n_+T)\cdot\na(n_+T) + \na\log(n_-T)\cdot\na(n_-T)
	\big)dx.
\end{align*}
This ends the proof.
\end{proof}

\begin{remark}\label{rem.ent1}\rm
When system \eqref{1.et1.n0}-\eqref{1.et1.vecn} includes the electric field,
a computation similar to the proof of Proposition \ref{prop.ent1}
shows that the entropy-production identity reads as
\begin{align*}
  \frac{dH_1}{dt} &+ \int_{\R^3}\bigg(
	\frac{|\na(n_+ T) + n_+T\na V|^2}{n_+T}
	+ \frac{|\na(n_- T) + n_-T\na V|^2}{n_-T}\bigg)dx \\
	&{}+ 10\int_{\R^3}(n_+ + n_-)|\na\sqrt{T}|^2\big)dx \\
	&{}+ \frac12\int_{\R^3}(n_+ - n_-)
  (\log n_+ - \log n_-)\bigg(\frac{1}{\tau_{\rm sf}} 
	+ T\bigg|\na\frac{\vec{n}}{|\vec{n}|}\bigg|^2\bigg)dx = 0.
\end{align*}
Thus, the presence of the electric field complicates the existence of a priori
bounds.
\qed
\end{remark}


\subsection{Entropy inequality for the second model}\label{sec.ent2}

We show that there exists an entropy 
for the second model \eqref{3.et2.n02}-\eqref{3.et2.vecw2} ($\vec{a}=0$)
for vanishing electric field,
\begin{equation}\label{41.m2}
  \pa_t n_0 = \Delta(n_0T), \quad \frac32\pa_t(n_0T) = \Delta Z_0, \quad
	\pa_t\vec{W} = \Delta \vec{Z} - \vec{\Omega}_{\rm e}\times\vec{W}
	- \frac{\vec{W}}{\tau_{\rm sf}},
\end{equation}
where $x\in\R^3$, $t>0$, and 
$(Z_0,\vec{Z})$ are defined in \eqref{32.Z0}-\eqref{32.vecZ}, i.e.
\begin{align*}
	Z_0 &= \frac{5}{18n_0}\big(W_+^{3/5}+W_-^{3/5}\big)\big(W_+^{7/5}+W_-^{7/5}\big), \\
  \vec{Z} &= \frac{5}{18n_0}\big(W_+^{3/5}+W_-^{3/5}\big)
	\big(W_+^{7/5}-W_-^{7/5}\big)\frac{\vec{W}}{|\vec{W}|}.
\end{align*}

We claim that the general entropy \eqref{1.H} 
becomes an entropy for the second model when
the Maxwellian $M=M_0\sigma_0+\vec{M}\cdot\vec{\sigma}$ is given by \eqref{32.M},
and this entropy equals, up to a constant, \eqref{1.H2}.
Note that if $\vec{W}=0$, we obtain $W_\pm=W_0=\frac32n_0T$ and 
$H_2=\int_{\R^d}n_0\log(n_0T^{-3/2})dx$, up to a constant. This function
corresponds to the entropy of the semiclassical energy-transport model
\cite[Chapter~6]{Jue09}. 

To show that \eqref{1.H} reduces to \eqref{1.H2},
we may employ \eqref{4.H} but we prefer to proceed in a slightly different way.
We observe that the Pauli matrices are traceless and we employ the formula  
$(\vec{c}\cdot\vec{\sigma})(\vec{M}\cdot\vec{\sigma})
= (\vec{c}\cdot\vec{M})\sigma_0 + i(\vec{c}\times\vec{M})\cdot\vec{\sigma}$
(see \cite[(7)]{PoNe11}) to infer that
\begin{align*}
  H_2 &= \frac12\operatorname{tr}\int_{\R^3}\int_{\R^3}
	(M_0+\vec{M}\cdot\vec{\sigma})
	\bigg(a_0\sigma_0 + (c_0\sigma_0+\vec{c}\cdot\vec{\sigma})
	\frac{|k|^2}{2}\bigg)dkdx \\
	&= \int_{\R^3}\int_{\R^3}\bigg(a_0M_0 + c_0M_0\frac{|k|^2}{2} 
	+ \vec{c}\cdot\vec{M}\frac{|k|^2}{2}\bigg)dkdx \\
	&= \int_{\R^3}\big(a_0n_0 + c_0W_0 + \vec{c}\cdot\vec{W}\big)dx.
\end{align*}
The Lagrange multiplier $a_0$ can be written in the following way, using
the first equation in \eqref{32.K} and \eqref{32.K2}:
$$
  a_0 = \log\frac{K}{(2\pi)^{3/2}} 
	= \frac52\log n_0 - \frac52\log(W_+^{3/5}+W_-^{3/5})
	+ \log(2\cdot 3^{3/2}(2\pi)^{-3/2}).
$$
Observing that $\int_{\R^3}n_0 dx$ is constant in time, it holds that, 
up to a constant,
$$
  \int_{\R^3}a_0n_0 dx = \frac52\int_{\R^3}n_0\log\frac{n_0}{W_+^{3/5}+W_-^{3/5}}dx.
$$
By the second equation in \eqref{32.K}, we have $c_0\pm|\vec{c}|=-1/\theta_\pm$,
which yields
$$
  c_0 = -\frac12\bigg(\frac{1}{\theta_+} + \frac{1}{\theta_-}\bigg), \quad
	|\vec{c}| = \frac12\bigg(\frac{1}{\theta_+} - \frac{1}{\theta_-}\bigg).
$$
Furthermore, employing the third equation in \eqref{32.K} and \eqref{32.gamma}, 
we have $\vec{c}/|\vec{c}|=\vec{\gamma}=\vec{W}/|\vec{W}|$ which shows that
$\vec{c}\cdot\vec{W}=|\vec{c}||\vec{W}|$. Thus, replacing $\theta_\pm$
by the expression in \eqref{32.thetapm},
\begin{align*}
  2(c_0W_0+\vec{c}\cdot\vec{W})
	&= -\bigg(\frac{1}{\theta_+} + \frac{1}{\theta_-}\bigg)W_0
	+ \bigg(\frac{1}{\theta_+} - \frac{1}{\theta_-}\bigg)|\vec{W}| \\
	&= -\frac{3n_0W_0}{W_+^{3/5}+W_-^{3/5}}\bigg(\frac{1}{W_+^{2/5}}
	+ \frac{1}{W_-^{2/5}}\bigg) + \frac{3n_0|\vec{W}|}{W_+^{3/5}+W_-^{3/5}}
	\bigg(\frac{1}{W_+^{2/5}} - \frac{1}{W_-^{2/5}}\bigg) \\
	&= -\frac{3n_0}{W_+^{3/5}+W_-^{3/5}}\bigg(\frac{W_0+|\vec{W}|}{W_+^{2/5}}
	+ \frac{W_0-|\vec{W}|}{W_-^{2/5}}\bigg) = -3n_0.
\end{align*}
Neglecting this contribution as well as the constant in the expression for $a_0$, 
this shows the claim.

We show now that the entropy \eqref{1.H2} is nonincreasing in time and
that it provides some gradient estimates.

\begin{proposition}[Entropy inequality for system \eqref{41.m2}]\label{prop.ent2}
The entropy \eqref{1.H2}, considered as a function of time, is nonincreasing 
along (smooth) solutions $(n_0,T,\vec{W})$ to \eqref{41.m2} in $\R^3$, 
where $W_\pm=W_0\pm|\vec{W}|$ and $W_0=\frac32 n_0T$. Furthermore, it holds
\begin{equation}\label{41.ent}
  \frac{dH_2}{dt} + c\int_{\R^3}\big(|\na\sqrt{W_+}|^2 + |\na\sqrt{W_-}|^2
	+ T|\na\sqrt{n_0}|^2 + W_0^{-1}|(\na \vec{W})^\top|^2\big) dx \le 0,
\end{equation}
where $c>0$ is a constant and $(\na\vec{W})^\top 
= ({\mathbb I}-|\vec{W}|^{-2}\vec{W}\otimes\vec{W})\na\vec{W}$.
\end{proposition}

\begin{proof}
First, we perform some auxiliary computations:
\begin{align*}
  \frac{\pa}{\pa n_0}\bigg(\frac52 n_0\log\frac{n_0}{W_+^{3/5}+W_-^{3/5}}\bigg)
	&= \frac52\bigg(\log n_0 - \log(W_+^{3/5}+W_-^{3/5}) + 1\bigg), \\
	\frac{\pa}{\pa W_0}\bigg(\frac52 n_0\log\frac{n_0}{W_+^{3/5}+W_-^{3/5}}\bigg)
	&= -\frac32 n_0\frac{W_+^{-2/5}+W_-^{-2/5}}{W_+^{3/5}+W_-^{3/5}}, \\
	\frac{\pa}{\pa \vec{W}}\bigg(\frac52 n_0\log\frac{n_0}{W_+^{3/5}+W_-^{3/5}}\bigg)
	&= -\frac32 n_0\frac{W_+^{-2/5}-W_-^{-2/5}}{W_+^{3/5}+W_-^{3/5}}
	\frac{\vec{W}}{|\vec{W}|}.
\end{align*}
Using these expressions and equations \eqref{41.m2}, it follows that
\begin{align*}
  \frac{dH_2}{dt} &= \int_{\R^3}\bigg\{\frac52\big(\log n_0 
	- \log(W_+^{3/5}+W_-^{3/5}) + 1\big)\pa_t n_0 \\
	&\phantom{xx}{}- \frac32n_0\frac{W_+^{-2/5}+W_-^{-2/5}}{W_+^{3/5}+W_-^{3/5}}\pa_t W_0
	- \frac32n_0\frac{W_+^{-2/5}-W_-^{-2/5}}{W_+^{3/5}+W_-^{3/5}}
	\frac{\vec{W}}{|\vec{W}|}\cdot\pa_t\vec{W}\bigg\}dx \\
	&= -\frac53\int_{\R^3}\na\log\frac{n_0}{W_+^{3/5}+W_-^{3/5}}\cdot\na W_0 dx \\
	&\phantom{xx}{}
	+ \frac{5}{12}\int_{\R^3}\na\frac{n_0(W_+^{-2/5}+W_-^{-2/5})}{W_+^{3/5}+W_-^{3/5}}
	\cdot\na\bigg(\frac{W_+^{3/5}+W_-^{3/5}}{n_0}(W_+^{7/5}+W_-^{7/5})\bigg)dx \\
	&\phantom{xx}{}+ \frac{5}{12}\int_{\R^3}\na\bigg(
	\frac{n_0(W_+^{-2/5}-W_-^{-2/5})}{W_+^{3/5}+W_-^{3/5}}\frac{\vec{W}}{|\vec{W}|}
	\bigg)\cdot\na\bigg(\frac{W_+^{3/5}+W_-^{3/5}}{n_0}(W_+^{7/5}-W_-^{7/5})
	\frac{\vec{W}}{|\vec{W}|}\bigg)dx \\
	&\phantom{xx}{}+ \frac{3}{2\tau_{\rm sf}}\int_{\R^3}
	\frac{n_0(W_+^{-2/5}-W_-^{-2/5})}{W_+^{3/5}+W_-^{3/5}}|\vec{W}|dx.
\end{align*}
Setting $\lambda:=n_0/(W_+^{3/5}+W_-^{3/5})$, we can rewrite $dH_2/dt$ as follows:
\begin{align*}
  \frac{dH_2}{dt} &= -\frac53\int_{\R^3}\lambda^{-1}\na\lambda\cdot\na W_0 dx \\
	&\phantom{xx}{}+ \frac{5}{12}\int_{\R^3}\na\big(\lambda(W_+^{-2/5}+W_-^{-2/5})\big)
	\cdot\na\big(\lambda^{-1}(W_+^{7/5}+W_-^{7/5})\big)dx \\
	&\phantom{xx}{}+ \frac{5}{12}\int_{\R^3}\na\bigg(\lambda(W_+^{-2/5}-W_-^{-2/5})
	\frac{\vec{W}}{|\vec{W}|}\bigg)\cdot\na\bigg(\lambda^{-1}(W_+^{7/5}-W_-^{7/5})
	\frac{\vec{W}}{|\vec{W}|}\bigg)dx \\
	&\phantom{xx}{}+ \frac{3}{2\tau_{\rm sf}}\int_{\R^3}\lambda|\vec{W}|
	(W_+^{-2/5}-W_-^{-2/5})dx \\
	&= I_1+I_2+I_3+I_4.
\end{align*}
Using $W_0=\frac12(W_++W_-)$, the first integral becomes
$$
  I_1 = -\frac56\int_{\R^3}\sum_{s=\pm}\frac{\na\lambda}{\lambda}\cdot\na W_s dx.
$$
By the product rule, the third integral $I_3$ is computed as
\begin{align*}
	I_3 &= \frac{5}{12}\int_{\R^3}\na\big(\lambda(W_+^{-2/5}-W_-^{-2/5})\big)
	\cdot\na\big(\lambda^{-1}(W_+^{7/5}-W_-^{7/5})\big)dx \\
	&\phantom{xx}{}+ \frac{5}{12}\int_{\R^3}\big(\lambda(W_+^{-2/5}-W_-^{-2/5})\big)
	\big(\lambda^{-1}(W_+^{7/5}-W_-^{7/5})\big)
	\bigg|\na\frac{\vec{W}}{|\vec{W}|}\bigg|^2 dx,
\end{align*}
where the mixed terms vanish since 
$\na(\vec{W}/|\vec{W}|)\cdot(\vec{W}/|\vec{W}|)=0$ (which is a
consequence of $\na|\vec{W}/|\vec{W}||^2=0$). 
Expanding the products in the first integral
on the right-hand side and in $I_2$, some terms cancel, and we end up with
\begin{align*}
  I_2+I_3 &= 
  \frac56\int_{\R^3}\big(\na(\lambda W_+^{-2/5})\cdot\na(\lambda^{-1}W_+^{7/5})
	+ \na(\lambda W_-^{-2/5})\cdot\na(\lambda^{-1}W_-^{7/5})\big)dx \\
  &\phantom{xx}{}+ \frac{5}{12}\int_{\R^3}(W_+^{-2/5}-W_-^{-2/5})
	(W_+^{7/5}-W_-^{7/5})\bigg|\na\frac{\vec{W}}{|\vec{W}|}\bigg|^2 dx \\
	&= \frac56\int_{\R^3}\sum_{s=\pm}\bigg(-\frac{14}{25}\frac{|\na W_s|^2}{W_s} 
	+ \frac95\frac{\na\lambda}{\lambda}\cdot\na W_s
	- \frac{W_s}{\lambda^2}|\na\lambda|^2\bigg)dx \\
	&\phantom{xx}{}+ \frac{5}{12}\int_{\R^3}(W_+^{-2/5}-W_-^{-2/5})
	(W_+^{7/5}-W_-^{7/5})\bigg|\na\frac{\vec{W}}{|\vec{W}|}\bigg|^2 dx.
\end{align*}
Finally, the fourth integral is nonpositive since $0\le W_-\le W_+$, i.e.\ $I_4\le 0$.
Combining these results, we find that
\begin{align}
  \frac{dH_2}{dt} &\le -\frac56\int_{\R^3}\sum_{s=\pm}\bigg(
	\frac{14}{25}\bigg|\frac{\na W_s}{\sqrt{W_s}}\bigg|^2
	- \frac45\frac{\na W_s}{\sqrt{W_s}}\cdot\frac{\sqrt{W_s}}{\lambda}\na\lambda
	+ \bigg|\frac{\sqrt{W_s}}{\lambda}\na\lambda\bigg|^2\bigg)dx \nonumber \\
	&\phantom{xx}{}- \frac{5}{12}\int_{\R^3}(W_-^{-2/5}-W_+^{-2/5})
	(W_+^{7/5}-W_-^{7/5})\bigg|\na\frac{\vec{W}}{|\vec{W}|}\bigg|^2 dx \nonumber \\
	&= J_1 + J_2. \label{41.J12} 
\end{align}

First, we consider the first integral $J_1$.
The quadratic form in $J_1$ is positive definite and the eigenvalues of the
associated matrix are larger than $1/5$, so
\begin{align}
  J_1 &\le -\frac15\int_{\R^3}\bigg(|\na\sqrt{W_+}|^2 + |\na\sqrt{W_-}|^2
	+ (W_+ + W_-)\bigg|\frac{\na\lambda}{\lambda}\bigg|^2\bigg)dx \nonumber \\
  &\le -\frac15\int_{\R^3}\bigg(|\na\sqrt{W_+}|^2 + |\na\sqrt{W_-}|^2
	+ 2\eps W_0\bigg|\frac{\na\lambda}{\lambda}\bigg|^2\bigg)dx, \label{41.J1}
\end{align}
where we replaced $W_+ + W_-$ by $2W_0$ and introduced some $\eps\in(0,1)$.
The last term can be reformulated in terms of $W_0=\frac12(W_+ + W_-)$ and 
$n_0=(W_+^{3/5}+W_-^{3/5})\lambda$, using the elementary inequalities
$(a-b)^2\ge \frac12 a^2-b^2$ and $-(a+b)^2\ge -2(a^2+b^2)$:
\begin{align*}
  W_0\bigg|\frac{\na\lambda}{\lambda}\bigg|^2
	&= W_0\bigg|\frac{\na n_0}{n_0} 
	- \frac35\frac{W_+^{-2/5}\na W_+}{W_+^{3/5}+W_-^{3/5}}
	- \frac35\frac{W_-^{-2/5}\na W_-}{W_+^{3/5}+W_-^{3/5}}\bigg|^2 \\
	&\ge \frac12 W_0\bigg|\frac{\na n_0}{n_0}\bigg|^2
	- \frac{18}{25}W_0\bigg(\frac{W_+^{-4/5}|\na W_+|^2}{(W_+^{3/5}+W_-^{3/5})^2}
	+ \frac{W_-^{-4/5}|\na W_-|^2}{(W_+^{3/5}+W_-^{3/5})^2}\bigg).
\end{align*}
Employing $W_+^{3/5}+W_-^{3/5}\ge (W_++W_-)^{3/5}=(2W_0)^{3/5}$ and 
$W_\pm^{1/5}\le (2W_0)^{1/5}$, we can estimate as follows:
\begin{align*}
  W_0\bigg|\frac{\na\lambda}{\lambda}\bigg|^2
	&\ge \frac12 W_0\bigg|\frac{\na n_0}{n_0}\bigg|^2
	- \frac{18}{25}W_0\bigg(\frac{4W_+^{1/5}|\na \sqrt{W_+}|^2}{(2W_0)^{6/5}}
	+ \frac{4W_-^{1/5}|\na W_-|^2}{(2W_0)^{6/5}}\bigg) \\
	&= \frac12 W_0\bigg|\frac{\na n_0}{n_0}\bigg|^2
	- \frac{18}{25}\frac{2^{4/5}}{W_0^{1/5}}\big(W_+^{1/5}|\na\sqrt{W_+}|^2
	+ W_-^{1/5}|\na\sqrt{W_-}|^2\big) \\
	&\ge \frac12 W_0\bigg|\frac{\na n_0}{n_0}\bigg|^2
	- \frac{36}{25}\big(|\na\sqrt{W_+}|^2	+ |\na\sqrt{W_-}|^2\big).
\end{align*}
Then, with the relation $W_0=\frac32 n_0T$,
$$
  W_0\bigg|\frac{\na\lambda}{\lambda}\bigg|^2
	\ge 3T|\na\sqrt{n_0}|^2
	- \frac{36}{25}\big(|\na\sqrt{W_+}|^2	+ |\na\sqrt{W_-}|^2\big).
$$
Inserting this expression into \eqref{41.J1} and choosing $\eps>0$
sufficiently small, we arrive at
\begin{equation}\label{41.J11}
  J_1 \le -c\int_{\R^3}\big(|\na\sqrt{W_+}|^2 + |\na\sqrt{W_-}|^2
	+ T|\na\sqrt{n_0}|^2\big)dx
\end{equation}
for some number $0<c<1/5$. 

Next, we estimate the second integral $J_2$ in \eqref{41.J12}. 
By the mean-value theorem, there exist $\xi$, $\eta\in[W_-,W_+]$ such that
\begin{align*}
  (W_-^{-2/5} & - W_+^{-2/5})(W_+^{7/5}-W_-^{7/5})
	= (W_+W_-)^{-2/5}(W_+^{2/5}-W_-^{2/5})(W_+^{7/5}-W_-^{7/5}) \\
	&= \frac{14}{25}(W_+W_-)^{-2/5}\xi^{-3/5}\eta^{2/5}(W_+ - W_-)^2 
	\ge \frac{14}{25}W_+^{-1}(W_+ - W_-)^2 
	\ge \frac{28}{25}\frac{|\vec{W}|^2}{W_0},
\end{align*}
where we used that $W_+\le 2W_0$ and $W_+ - W_-=2|\vec{W}|$. Consequently,
$$
  J_2 \ge \frac{28}{25}\int_{\R^3}\frac{|\vec{W}|^2}{W_0}
	\bigg|\na\frac{\vec{W}}{|\vec{W}|}\bigg|^2 dx
	= \frac{28}{25}\int_{\R^3}\frac{1}{W_0}
	\bigg|\bigg({\mathbb I}-\frac{\vec{W}\otimes\vec{W}}{|\vec{W}|^2}
	\bigg)\na\vec{W}\bigg|^2 dx.
$$
Combining this inequality and \eqref{41.J11} with \eqref{41.J12},
the result follows.
\end{proof}


\subsection{Entropy inequality for the third model}\label{sec.ent3}

We show that there exists an entropy for the third model 
\eqref{1.et3.npm}-\eqref{1.et3.s} ($\vec{a}=\lambda\vec{c}$) for vanishing
electric fields, i.e.
\begin{align}
  & \pa_t n_\pm - \Delta(n_\pm T_\pm) = \mp\frac{1}{2\tau_{\rm sf}}(n_+ - n_-)
	\mp \frac12(n_+T_+ - n_-T_-)|\na\vec{s}|^2, \label{43.eq1} \\
	& \frac32\pa_t(n_\pm T_\pm) - \frac52\Delta(n_\pm T_\pm^2)
	= \mp\frac{3}{4\tau_{\rm sf}}(n_+T_+ - n_-T_-)
	\mp \frac54(n_+T_+^2 - n_-T_-^2)|\na\vec{s}|^2, \label{43.eq2} \\
  & \pa_t\vec{s} - \frac{n_+T_+ - n_-T_-}{n_+ - n_-}\vec{s}\times
	(\Delta\vec{s}\times\vec{s}) = 2\frac{\na(n_+T_+ - n_-T_-)}{n_+ - n_-}
	\cdot\na\vec{s} - \vec{\Omega}_{\rm e}\times\vec{s}, \label{43.eq3}
\end{align}
where $x\in\R^3$, $t>0$. As in Section \ref{sec.ent2}, we make first explicit
the entropy functional \eqref{1.H}, where the Maxwellian is given by its
Pauli components \eqref{3.max3}. A computation shows that
$M_\pm=n_\pm g_{T_\pm}(k)\sigma_0$, so \eqref{4.H} yields immediately \eqref{1.H3}.

\begin{proposition}[Entropy inequality for system \eqref{43.eq1}-\eqref{43.eq3}]
\label{prop.ent3}
The entropy \eqref{1.H3}, considered as a function of time, is nonincreasing 
along (smooth) solutions $(n_\pm,T_\pm,\vec{s})$ to \eqref{43.eq1}-\eqref{43.eq3}
in $\R^3$, and there exists a number $c>0$ such that
$$
  \frac{dH_3}{dt} + c\int_{\R^3}\sum_{s=\pm}\big(T_s|\na\sqrt{n_s}|^2
	+ n_s|\na\sqrt{T_s}|^2\big)dx \le 0.
$$
\end{proposition}

\begin{proof}
Before computing the derivative $dH_3/dt$, let us consider the semiclassical
energy-transport system
$$
  \pa_t n = \Delta(nT), \quad \frac32\pa_t(nT) = \frac52\Delta(nT^2)
	\quad\mbox{in }\R^3,
$$
which is known to dissipate the entropy $H_0=\int_{\R^3}n\log(nT^{-3/2})dx$.
Indeed, a computation shows that
\begin{align*}
  \frac{dH_0}{dt} &= -\int_{\R^3}\bigg(\na(nT)\cdot\na(nT^{-3/2})
	- \frac52\na\frac{1}{T}\cdot\na(nT^2)\bigg)dx \\
	&= -4\int_{\R^3}\bigg(|\sqrt{T}\na\sqrt{n}|^2 + 2\sqrt{T}\na\sqrt{n}\cdot
	\sqrt{n}\na\sqrt{T} + \frac72|\sqrt{n}\sqrt{T}|^2\bigg)dx.
\end{align*}
The quadratic form in the variables $\sqrt{T}\na\sqrt{n}$ and 
$\sqrt{n}\na\sqrt{T}$ is positive definite and the eigenvalues of the associated
matrix are larger than $1/2$, so
$$
  \frac{dH_0}{dt} + 2\int_{\R^3}\big(|\sqrt{T}\na\sqrt{n}|^2
	+ |\sqrt{n}\na\sqrt{T}|^2\big)dx \le 0.
$$

The similarity in structure between $H_3$ and $H_0$ as well as between the
spin and semiclassical energy-transport system allows us to deduce that, for some
number $c>0$,
\begin{align*}
  \frac{dH_3}{dt} & + c\int_{\R^3}\sum_{s=\pm}\big(T_s|\na\sqrt{n_s}|^2
	+ n_s|\na T_s|^2\big)dx \\
	&\le - \frac{1}{2\tau_{\rm sf}}
	\int_{\R^3}\bigg(\log\frac{n_+T_+^{-3/2}}{n_-T_-^{-3/2}}(n_+ - n_-)
	+ \frac32\frac{T_+ - T_-}{T_+T_-}(n_+T_+ - n_-T_-)\bigg)dx \\
	&\phantom{xx}{}- \frac12\int_{\R^3}\bigg(\log\frac{n_+T_+^{-3/2}}{n_-T_-^{-3/2}}
	(n_+T_+ - n_-T_-)
	+ \frac52\frac{T_+ - T_-}{T_+T_-}(n_+T_+^2 - n_-T_-^2)\bigg)|\vec{s}|^2dx \\
  &= I_1 + I_2. 
\end{align*}
We claim that $I_1\le 0$ and $I_2\le 0$ which concludes the proof.

First, we prove that $I_1\le 0$. It holds that
\begin{align*}
  I_1 &= -\frac{1}{2\tau_{\rm sf}}\int_{\R^3}(n_+ - n_-)(\log n_+ - \log n_-)dx \\
	&\phantom{xx}{}- \frac{3}{4\tau_{\rm sf}}
	\int_{\R^3}\bigg(-(n_+ - n_-)\log\frac{T_+}{T_-}
	+ \frac{T_+ - T_-}{T_+T_-}(n_+T_+ - n_-T_-)\bigg)dx \\
	&\le -\frac{3}{4\tau_{\rm sf}}\int_{\R^3}\bigg(-(n_+ - n_-)\log\frac{T_+}{T_-}
	+ \frac{T_+ - T_-}{T_+T_-}(n_+T_+ - n_-T_-)\bigg)dx.
\end{align*}
Because of $T_-\le T_+$ and $n_-\ge 0$, we have $n_+T_+ - n_-T_-\ge (n_+ - n_-)T_+$
which shows that
\begin{align*}
  I_1 &\-e -\frac{3}{4\tau_{\rm sf}}\int_{\R^3}\bigg(-(n_+ - n_-)\log\frac{T_+}{T_-}
	+ \frac{1}{T_-}(n_+ - n_-)(T_+ - T_-)\bigg)dx \\
	&= -\frac{3}{4\tau_{\rm sf}}
	\int_{\R^3}(n_+ - n_-)\bigg(\frac{T_+}{T_-} - 1 - \log\frac{T_+}{T_-}
	\bigg)dx \le 0.
\end{align*}
In a similar way as above, we find that $n_+T_+^2 - n_-T_-^2\ge (n_+T_+ - n_-T_-)T_+$
and
\begin{align*}
  I_2 &= -\frac12\int_{\R^3}(n_+T_+ - n_-T_-)(\log(n_+T_+)-\log(n_-T_-))dx \\
	&\phantom{xx}{}- \frac54\int_{\R^3}\bigg(-(n_+T_+ - n_-T_-)\log\frac{T_+}{T_-}
	+ \frac{T_+ - T_-}{T_+T_-}(n_+T_+^2 - n_-T_-^2)\bigg)|\vec{s}|^2dx \\
  &\le -\frac54\int_{\R^3}\bigg(-(n_+T_+ - n_-T_-)\log\frac{T_+}{T_-}
	+ \frac{1}{T_-}(n_+T_+ - n_-T_-)(T_+ - T_-)\bigg)|\vec{s}|^2dx \\
	&= -\frac54\int_{\R^3}(n_+T_+ - n_-T_-)
	\bigg(\frac{T_+}{T_-} - 1 - \log\frac{T_+}{T_-}\bigg)dx \le 0.
\end{align*}
This finishes the proof.
\end{proof}


\section{Existence analysis of the second model}\label{sec.ex}

We show the existence of weak solutions to a time-discrete version
of the second model in the formulation \eqref{3.et2.n02}-\eqref{3.et2.vecw2} 
for vanishing electric field. Replacing $W_0=\frac32 n_0T$ and $Z_0$, $\vec{Z}$ by
\eqref{32.Z0}, \eqref{32.vecZ}, respectively, we obtain system
\eqref{12.n0}-\eqref{12.vecw}. We recall that $h>0$ is the time step 
size, $(n_0,W_0,\vec{W})$ are the unknowns, and $(n_0^0,W_0^0,\vec{W}^0)$ 
are the moments at the previous time step (supposed to be given). 

\begin{theorem}[Existence for the time-discrete second model]\label{thm.ex}
Let $\dom\subset\R^d$ ($d\le 3$) be a bounded domain and let $n_0^0$, 
$W_0^0\in L^2(\dom)$, 
$n_0^D$, $W_0^D\in H^1(\dom)\cap L^\infty(\dom)$, 
$\vec{W}^D\in H^1(\dom;\R^3)\cap L^\infty(\dom;\R^3)$, 
$\vec{W}^0\in L^2(\dom;\R^3)$
satisfy $\sup_{\pa \dom}|\vec{W}^D|/W_0^D<1$ and
$$
  n_0^0>0, \ W_0^0>0\quad\mbox{in }\dom, \quad \inf_{D}\frac{W_0^0}{n_0^0}>0,
	\quad \sup_{\dom}\frac{|\vec{W}^0|}{W_0^0} < 1.
$$
Then there exists a solution $(n_0,W_0,\vec{W})\in H^1(\dom;\R^5)$ to
\eqref{12.n0}-\eqref{12.bc} such that 
$$
  n_0>0, \ W_0>0\quad\mbox{in }\dom, \quad \inf_{D}\frac{W_0}{n_0}>0,
	\quad \sup_{\dom}\frac{|\vec{W}|}{W_0} < 1.
$$
\end{theorem}

\begin{proof}
The proof is inspired by the techniques employed in \cite{ZaJu15}.
The idea is to introduce new variables to make the differential operator 
linear and to truncate the nonlinearities. We proceed in several steps.

{\em Step 1: new variables.} Let $W_\pm=W_0\pm|\vec{W}|$. We define 
\begin{align*}
  & u := \frac23 W_0, \quad 
	v_0 := \frac{5}{18n_0}(W_+^{3/5}+W_-^{3/5})(W_+^{7/5}+W_-^{7/5}), \\
	& \vec{v} := \frac{5}{18n_0}(W_+^{3/5}+W_-^{3/5})(W_+^{7/5}-W_-^{7/5})
  \frac{\vec{W}}{|\vec{W}|}.
\end{align*}
Observe that $\sup_{\pa \dom}|\vec{W}^D|/W_0^D<1$ implies that
$\inf_{\pa \dom}W_\pm>0$ and $\sup_{\pa \dom}|\vec{v}|/v_0<1$.
Furthermore,
\begin{align*}
  & |\vec{v}| = \frac{5}{18n_0}(W_+^{3/5}+W_-^{3/5})(W_+^{7/5}-W_-^{7/5}), \quad
	\frac{\vec{v}}{|\vec{v}|} = \frac{\vec{W}}{|\vec{W}|}, \\
  & v_\pm := v_0\pm|\vec{v}| = \frac{5}{9n_0}(W_+^{3/5}+W_-^{3/5})W_\pm^{7/5}.
\end{align*}
This shows that $v_+/v_- = (W_+/W_-)^{7/5}$ or equivalently,
$W_+/W_-=(v_+/v_-)^{5/7}$. We rewrite the variables in terms of $v_\pm$,
observing that $v_+ + v_- = 2v_0$:
\begin{align}
  v_0 &= \frac{5}{18}\frac{W_-^2}{n_0}\bigg(1+\bigg(\frac{v_+}{v_-}\bigg)^{3/7}
	\bigg)\bigg(1+\frac{v_+}{v_-}\bigg)
	= \frac{5}{18}\frac{W_+^2}{n_0}\bigg(1+\bigg(\frac{v_-}{v_+}\bigg)^{3/7}
	\bigg)\bigg(1+\frac{v_-}{v_+}\bigg), \nonumber \\
	W_\pm &= \bigg(\frac{18}{5}n_0v_0\bigg)^{1/2}v_\pm^{5/7}
	(v_+^{3/7}+v_-^{3/7})^{-1/2}(v_++v_-)^{-1/2} \label{42.Wpm}\\
	&= \bigg(\frac95 n_0\bigg)^{1/2}v_\pm^{5/7}(v_+^{3/7}+v_-^{3/7})^{-1/2}
	\nonumber \\
	u &= \frac13(W_++W_-) = \bigg(\frac25n_0v_0\bigg)^{1/2}
	(v_+^{5/7}+v_-^{5/7})(v_+^{3/7}+v_-^{3/7})^{-1/2}(v_+ + v_-)^{1/2} \nonumber \\
	&= \bigg(\frac{n_0}{5}\bigg)^{1/2}
	(v_+^{5/7}+v_-^{5/7})(v_+^{3/7}+v_-^{3/7})^{-1/2}. \nonumber
\end{align}
Solving the last expression for $n_0$ yields
\begin{equation}\label{42.fct.n0}
  n_0 = 5u^2\frac{v_+^{3/7}+v_-^{3/7}}{(v_+^{5/7}+v_-^{5/7})^2},
\end{equation}
and inserting this equation into \eqref{42.Wpm} gives
$W_\pm = 3uv_\pm^{5/7}/(v_+^{5/7}+v_-^{5/7})$.
Because of $\vec{v}/|\vec{v}|=\vec{W}/|\vec{W}|$, it follows that
\begin{equation}\label{42.fct.w0}
  W_0 = \frac12(W_+ + W_-) = \frac32 u, \quad
	\vec{W} = \frac32u\frac{v_+^{5/7}-v_-^{5/7}}{v_+^{5/7}+v_-^{5/7}}
	\frac{\vec{v}}{|\vec{v}|}.
\end{equation}
We infer that system \eqref{12.n0}-\eqref{12.vecw} can be written as
\begin{align}
  n_0(u,v_0,\vec{v}) - h\Delta u &= n_0^0, \label{42.n02} \\
	W_0(u,v_0,\vec{v}) - h\Delta v_0 &= W_0^0, \label{42.w02} \\
	\bigg(1+\frac{h}{\tau_{\rm sf}}\bigg)\vec{W}(u,v_0,\vec{v})
	- h\Delta\vec{v} &= \vec{W}^0\quad\mbox{in }\dom,
\end{align}
where $n_0(u,v_0,\vec{v})$, $W_0(u,v_0,\vec{v})$, and 
$\vec{W}(u,v_0,\vec{v})$ are given by \eqref{42.fct.n0}-\eqref{42.fct.w0}.

{\em Step 2: truncation.} We introduce for $\eps>0$ the truncation operator
$$
  [f]_\eps := \left\{\begin{array}{ll}
	0 &\quad\mbox{for }f\le 0, \\
	f &\quad\mbox{for }0<f\le 1/\eps, \\
	1/\eps &\quad\mbox{for }f>1/\eps,
	\end{array}\right.
$$
and the auxiliary functions
\begin{align*}
  \lambda(\xi,v_+,v_-) 
	&:= \frac52\xi\frac{(v_+^{3/7}+v_-^{3/7})(v_+ + v_-)}{(v_+^{5/7}+v_-^{5/7})^2}, \\
	\mu(\xi,v_+,v_-) 
	&:= \frac32\bigg(1+\frac{h}{\tau_{\rm sf}}\bigg)\xi
	\frac{v_+^{5/7}-v_-^{5/7}}{v_+^{5/7}+v_-^{5/7}}.
\end{align*}
These definitions imply that
$$
  n_0(u,v_0,\vec{v}) = \lambda(u/v_0,v_+,v_-)u, \quad
	\vec{W}(u,v_0,\vec{v}) = \mu(u/v_0,v_+,v_-)v_0\frac{\vec{v}}{|\vec{v}|}.
$$

We claim that the following estimate holds for $\lambda$ and $\mu$:
\begin{equation}\label{42.lamu}
  \frac52\xi\le\lambda(\xi,v_+,v_-) \le 6\xi, \quad
	0\le\mu(\xi,v_+,v_-) \le\frac32\bigg(1+\frac{h}{\tau_{\rm sf}}\bigg)\xi
\end{equation}
for all $\xi\ge 0$, $v_+\ge v_-\ge 0$.
Indeed, the bounds for $\mu$ are obvious. In order to prove the upper bound for
$\lambda$, we observe that
$$
  (v_+^{3/7}+v_-^{3/7})(v_+ + v_-)
	= v_+^{10/7} + v_-^{10/7} + v_+^{3/7}v_- + v_-^{3/7}v_+.
$$
By Young's inequality,
$$
  v_+^{3/7}v_- \le \frac{3}{10}v_+^{10/7} + \frac{7}{10}v_-^{10/7}
	\le \frac{7}{10}\big(v_+^{10/7} + v_-^{10/7}\big),
$$
and the same bound holds for $v_-^{3/7}v_+$ such that
$$
  (v_+^{3/7}+v_-^{3/7})(v_+ + v_-) \le \frac{12}{5}\big(v_+^{10/7} + v_-^{10/7}\big)
	\le \frac{12}{5}\big(v_+^{5/7} + v_-^{5/7}\big)^2.
$$
Inserting this estimate into the definition of $\lambda$, the upper bound
follows. The lower bound is equivalent to
$(v_+^{3/7}+v_-^{3/7})(v_+ + v_-)\ge (v_+^{5/7}+v_-^{5/7})^2$ which follows from
\begin{align*}
  \big(v_+^{5/7}&+v_-^{5/7}\big)^2 - \big(v_+^{3/7}+v_-^{3/7}\big)(v_+ + v_-)
	= 2v_+^{5/7}v_-^{5/7} - v_+^{3/7}v_- - v_-^{3/7}v_+ \\
	&= 2(v_+v_-)^{5/7}\bigg(1 - \frac12\bigg(\frac{v_-}{v_+}\bigg)^{2/7}
	- \frac12\bigg(\frac{v_+}{v_-}\bigg)^{2/7}\bigg) \le 0.
\end{align*}
This completes the proof of \eqref{42.lamu}.

With the above truncation, we wish to prove the existence of a weak solution
to
\begin{align}
  \lambda\big([u/v_0]_\eps,v_+,v_-\big)u - h\Delta u &= n_0^0, \label{42.eps1} \\
	\frac32[u/v_0]_\eps v_0 - h\Delta v_0 &= W_0^0, \label{42.eps2} \\
	\mu\big([u/v_0]_\eps,v_+,v_-\big)v_0\frac{\vec{v}}{|\vec{v}|} - h\Delta\vec{v}
	&= \vec{W}^0\quad\mbox{in }\dom, \label{42.eps3}
\end{align}
where, slightly abusing the notation, $v_\pm$ is here defined by
$v_\pm=\max\{0,v_0\pm|\vec{v}|\}$.
Since we will prove below that $v_0\pm|\vec{v}|\ge 0$, this notation is consistent.
The boundary conditions are
\begin{equation}\label{42.bc}
  u = u^D:=\frac23 W_0^D, \quad v_0 = v_0^D, \quad \vec{v}=\vec{v}^D
	\quad \mbox{on }\pa \dom,
\end{equation}
where
\begin{align*}
  v_0^D &:= \frac{5}{18n_0^D}((W_+^D)^{3/5}+(W_-^D)^{3/5})
	((W_+^D)^{7/5}+(W_-^D)^{7/5}), \\
	\vec{v}^D &:= \frac{5}{18n_0^D}((W_+^D)^{3/5}+(W_-^D)^{3/5})
	((W_+^D)^{7/5}-(W_-^D)^{7/5})\frac{\vec{W}^D}{|\vec{W}^D|},
\end{align*}
and $W_\pm^D:=W_0^D\pm|\vec{W}^D|$. 

{\em Step 3: existence of solutions to the truncated problem.}
The existence of a solution to \eqref{42.eps1}-\eqref{42.eps3} is shown using
the Leray-Schauder fixed-point theorem. For this, we define the mapping
$F:L^2(\dom;\R^5)\times[0,1]\to L^2(\dom;\R^5)$, $F(\rho,\nu_0,\vec{\nu};\sigma)
= (u/v_0,v_0,\vec{v})$, auch that
\begin{align}
  \sigma\lambda\big([\rho]_\eps,\nu_+,\nu_-\big)u - h\Delta u &= n_0^0
	\quad\mbox{in }\dom, \quad u=u^D\quad\mbox{on }\pa \dom, \label{42.lin1} \\
	\frac32\sigma[\rho]_\eps v_0 - h\Delta v_0 &= W_0^0
	\quad\mbox{in }\dom, \quad v_0=v_0^D\quad\mbox{on }\pa \dom, \label{42.lin2} \\
	\sigma\mu\big([\rho]_\eps,\nu_+,\nu_-\big)\nu_0
	\frac{\vec{\nu}}{|\vec{\nu}|} - h\Delta \vec{v} &= \vec{W}^0
	\quad\mbox{in }\dom, \quad \vec{v}=\vec{v}^D\quad\mbox{on }\pa \dom, 
	\label{42.lin3}
\end{align}
where $\nu_\pm := \max\{0,\nu_0\pm|\vec{\nu}|\}$. We first show that
$F$ is well defined, i.e.\ $u/v_0\in L^2(\dom)$. Standard eliptic regularity
implies that $u$, $v_0\in H^2(D)\subset L^\infty(\dom)$ (here we use $d\le 3$).
By Stampacchia's truncation technique, we infer that $u$ and $v_0$ are
strictly positive (see, e.g., Step 2 in \cite[Section~2]{ZaJu15}). We deduce
that $u/v_0\in H^1(\dom)$, and $F$ is well defined. Since $\vec{v}\in H^2(\dom;\R^3)$
by elliptic regularity again, the range of $F$ lies in $H^1(\dom;\R^5)$.
Employing $u$, $v_0$, $\vec{v}$, respectively, as test functions in the weak
formulation of \eqref{42.lin1}-\eqref{42.lin3} and using the Poincar\'e
inequality (note that $(u,v_0,\vec{v})$ are bounded functions), 
we obtain for some constant $C>0$,
$$
  \|F(\rho,\nu_0,\vec{\nu};\sigma)\|_{H^1(\dom;\R^5)}
	\le C\big(\|(n_0^0,W_0^0,\vec{W}^0)\|_{L^2(\dom;\R^5)}
	+ \|(\rho,\nu_0,\vec{\nu})\|_{L^2(\dom;\R^5)}\big).
$$
Standard arguments show that $F$ is continuous. Then the Sobolev embedding 
$H^1(\dom)\hookrightarrow L^2(\dom)$ implies that $F$ is compact.
Moreover, $F(\cdot;0)$ is constant.

It remains to derive uniform a priori estimates for all fixed points of 
$F(\cdot;\sigma)$. Let $(\rho,\nu_0,\vec{\nu})\in L^2(\dom;\R^5)$ be such a fixed 
point. Then $(\rho,\nu_0,\vec{\nu})\in H^1(\dom;\R^5)$ and $u=\rho \nu_0$. 
Employing $u-u^D$, $v_0-v_0^D$ as test functions in the weak formulation of 
\eqref{42.lin1}-\eqref{42.lin2}, respectively, and the Poincar\'e
inequality, we find that
$$
  \|u\|_{H^1(\dom)} + \|v_0\|_{H^1(\dom)} 
	\le ch^{-1/2}\big(\|n_0^0\|_{L^2(\dom)} + \|W_0^0\|_{L^2(\dom)}\big),
$$
where here and in the following, $c>0$ denotes a generic constant independent
of the solutions (and of $\eps$).
Similarly, with the test function $\vec{v}-\vec{v}^D$ in the weak formulation of 
\eqref{42.lin3}, using the nonnegativity of $\nu_0$ and $\mu$,
$$
  \|\vec{v}\|_{H^1(\dom)} \le Ch^{-1/2}\|\vec{W}^0\|_{L^2(\dom)}.
$$
These estimates provides the uniform bound in $L^2(\dom;\R^5)$ for all
fixed points of $F(\cdot,\sigma)$. By the Leray-Schauder fixed-point theorem, 
we infer the existence of a weak solution to \eqref{42.eps1}-\eqref{42.eps3}.

{\em Step 4: removing the truncation.} We prove that that there exists
a positive lower bound for $u/v_0$ which is independent of $\eps$.
As a consequence, the truncation in \eqref{42.eps1}-\eqref{42.eps3} can
be removed for sufficiently small values of $\eps>0$, giving a solution
to \eqref{12.n0}-\eqref{12.vecw}.

We choose $\eps:=\min\{\inf_{\dom}(W_0^0/n_0^0),(\sup_\dom(u^D/v_0^D))^{-1}\}$, 
which is positive by assumption,
and define $\phi(z)=\max\{0,z-1/\eps\}$. We use the (admissible) test functions
$v_0\phi(u/v_0)$, $u\phi(u/v_0)$ in \eqref{42.eps1}, \eqref{42.eps2},
respectively, and take the difference of the resulting equations:
\begin{align}
  \int_{\dom} & \bigg(\lambda\big([u/v_0]_\eps,v_+,v_-\big)
	- \frac32[u/v_0]_\eps\bigg)uv_0\phi(u/v_0)dx \nonumber  \\
	&\phantom{xx}{}+ h\int_{\dom}\big(\na(v_0\phi(u/v_0))\cdot\na u
	- \na(u\phi(u/v_0))\cdot\na v_0\big)dx \label{42.aux}\\
	&= \int_{\dom}(v_0n_0^0-uW_0^0)\phi(u/v_0)dx. \nonumber
\end{align}
By \eqref{42.lamu}, the first integral on the left-hand side can be estimated from
below,
$$
  \int_{\dom} \bigg(\lambda\big([u/v_0]_\eps,v_+,v_-\big)
	- \frac32[u/v_0]_\eps\bigg)uv_0\phi(u/v_0)dx 
	\ge \int_{\dom}[u/v_0]_\eps uv_0\phi(u/v_0)dx \ge 0.
$$
The second integral on the left-hand side of \eqref{42.aux} is nonnegative since
\begin{align*}
  \int_{\dom} & \big(\na(v_0\phi(u/v_0))\cdot\na u
	- \na(u\phi(u/v_0))\cdot\na v_0\big)dx \\
	&= \int_{\dom}\big(v_0\na\phi(u/v_0)\cdot\na u - u\na\phi(u/v_0)\cdot\na v_0
	\big)dx \\
	&= \int_{\dom}(v_0\na u-u\na v_0)\cdot\na(u/v_0)\phi'(u/v_0)dx
	= \int_{\dom}v_0^2|\na(u/v_0)|^2\phi'(u/v_0)dx \ge 0.
\end{align*}
Finally, because of $\phi(u/v_0)=0$ if $v_0/u\le\eps$ and $\eps\le W_0^0/n_0^0$, 
the integral on the right-hand side of \eqref{42.aux} becomes
\begin{align*}
  \int_{\dom}(v_0n_0^0-uW_0^0)\phi(u/v_0)dx
	&= \int_{\dom}un_0^0\bigg(\frac{v_0}{u}-\frac{W_0^0}{n_0^0}\bigg)\phi(u/v_0)dx \\
	&\le \int_{\dom}un_0^0\bigg(\eps-\frac{W_0^0}{n_0^0}\bigg)\phi(u/v_0)dx \le 0.
\end{align*}
Therefore, \eqref{42.aux} implies that
$$
  0\le \int_{\dom}uv_0\phi(u/v_0)dx\le 0,
$$
from which we deduce that $\phi(u/v_0)=0$ a.e.\ in $\dom$ and consequently,
$u/v_0\le 1/\eps$ a.e.\ in $\dom$. Hence, $[u/v_0]_\eps=u/v_0$ and we can remove
the truncation in \eqref{42.eps1}-\eqref{42.eps3}.

{\em Step 5: proof of $v_0>|\vec{v}|$ in $\dom$.}
More precisely, we show that $(1-\delta)v_0-|\vec{v}|\ge 0$ for
$$
  0 < \delta < \min\bigg\{1-\sup_{\dom}\frac{|\vec{W}^0|}{W_0^0},
	1-\sup_\dom\frac{|\vec{v}^D|}{v_0^D},\bigg(\frac16\bigg)^{7/5}\bigg\}.
$$
Note that such a choice is possible because of our assumptions.
To prove the claim, we use $w:=\min\{0,(1-\delta)v_0-|\vec{v}|\}$ and
$\vec{w}:=w\vec{v}/|\vec{v}|$ as test functions
in the weak formulations of \eqref{42.eps2} and \eqref{42.eps3}, respectively.
Note that, since $v_0$ is strictly positive, $\vec{w}$ vanishes in a neighborhood
of $\vec{v}=0$, so $\vec{w}\in H^1(\dom)$. By definition of $\delta$, it holds that
$w=0$ on $\pa \dom$, so $w$, $\vec{w}\in H_0^1(\dom)$.
We find that
\begin{align*}
  \frac32\int_{\dom}uw dx + h\int_{D}\na v_0\cdot\na w dx
	&= \int_{\dom}W_0^0wdx, \\
	\frac32\bigg(1+\frac{h}{\tau_{\rm sf}}\bigg)\int_{\dom}
	\frac{v_+^{5/7}-v_-^{5/7}}{v_+^{5/7}+v_-^{5/7}}uwdx
	+ h\int_{\dom}\na\vec{v}\cdot\na\vec{w}dx
	&= \int_{\dom}\vec{W}^0\cdot\vec{w}dx.
\end{align*}
We take the difference between the first equation, multiplied by $1-\delta$,
and the second equation:
\begin{align}
  \int_{\dom} & \bigg((1-\delta)-\frac32\bigg(1+\frac{h}{\tau_{\rm sf}}\bigg)
	\frac{v_+^{5/7}-v_-^{5/7}}{v_+^{5/7}+v_-^{5/7}}\bigg)uw dx \label{42.aux2} \\
	&{}+ h\int_{\dom}\big((1-\delta)\na v_0\cdot\na w - \na\vec{v}\cdot\na\vec{w}\big)dx
	= \int_{\dom}\bigg((1-\delta)W_0^0 - \frac{\vec{v}}{|\vec{v}|}\cdot\vec{W}^0
	\bigg) w dx. \nonumber
\end{align}
We deduce from the definition of $\delta$ that for any $z\ge 1-\delta$,
$$
  \frac{(1+z)^{5/7}-\max\{0,1-z\}^{5/7}}{(1+z)^{5/7}+\max\{0,1-z\}^{5/7}}
	= 1 - \frac{2\max\{0,1-z\}^{5/7}}{(1+z)^{5/7}+\max\{0,1-z\}^{5/7}}
	\ge 1 - 2\delta^{5/7} > \frac23.
$$
Thus, since $v_\pm=\max\{0,v_0\pm|\vec{v}|\}$ and taking $z=|\vec{v}|/v_0\ge
1-\delta$ on $\{w\le 0\}$, the first integral on the
left-hand side of \eqref{42.aux2} is estimated as
\begin{align*}
  \int_{\dom} & \bigg((1-\delta)-\frac32\bigg(1+\frac{h}{\tau_{\rm sf}}\bigg)
	\frac{v_+^{5/7}-v_-^{5/7}}{v_+^{5/7}+v_-^{5/7}}\bigg)uw dx \\
	&\ge \int_{\dom}\bigg(1 - \frac32\frac{(1+|\vec{v}|/v_0)^{5/7}
	-\max\{0,1-|\vec{v}|/v_0\}^{5/7}}{(1+|\vec{v}|/v_0)^{5/7}
	+ \max\{0,1-|\vec{v}|/v_0\}^{5/7}}\bigg)uwdx \\
	&\ge -c_\delta\int_{\dom}uwdx
  = -c_\delta\int_{\dom}u\max\{0,(1-\delta)v_0-|\vec{v}|\}dx,
\end{align*}
where $c_\delta=\frac32(1-2\delta^{5/7})-1=\frac12-3\delta^{5/7}>0$.
The second integral on the left-hand side of \eqref{42.aux2} equals
\begin{align*}
  \int_{\dom} & \bigg((1-\delta)\na v_0\cdot w - \na\vec{v}\cdot
	\na\bigg(w\frac{\vec{v}}{|\vec{v}|}\bigg)\bigg)dx \\
	&= \int_{\dom}\bigg((1-\delta)\na v_0\cdot w - \na|\vec{v}|\cdot\na w
	- w\na\vec{v}\cdot\na\frac{\vec{v}}{|\vec{v}|}\bigg)dx \\
	&= \int_{\dom}\bigg(|\na w|^2 - |\vec{v}|w\bigg|\na\frac{\vec{v}}{|\vec{v}|}\bigg|^2
	\bigg)dx \ge 0,
\end{align*}
using the fact that $w\le 0$. Finally, by the definition of $\delta$,
the integral on the right-hand side of \eqref{42.aux2} is nonpositive,
$$
  \int_{\dom} \bigg((1-\delta)W_0^0 - \frac{\vec{v}}{|\vec{v}|}\cdot\vec{W}^0
	\bigg) w dx 
	\le \int_{\dom}\big((1-\delta)W_0^0 - |\vec{W}^0|\big) w dx \le 0.
$$
Summarizing these estimates, \eqref{42.aux2} implies that
$$
  -c_\delta\int_{\dom}u\min\{0,(1-\delta)v_0-|\vec{v}|\}dx 
	= -c_\delta\int_{\dom}uw dx \le 0
$$
and hence, $(1-\delta)v_0-|\vec{v}|\ge 0$ a.e.\ in $\dom$, which proves the claim.
\end{proof}


\section{Numerical experiments}\label{sec.num}

We perform some numerical simulations using the first model 
\eqref{1.et1a.n0}-\eqref{1.et1a.vecn} with the spin polarization matrix,
We consider, as in \cite{PoNe11},
three- and five-layer structures that consist of alternating nonmagnetic and
ferromagnetic layers. Multilayer structures are promising for applications in
micro-sensor and high-frequency devices. In this paper, they serve to illustrate
the solution behavior rather than to model practical devices.

\subsection{Numerical scheme}

We solve equations \eqref{1.et1a.n0}-\eqref{1.et1a.vecn}
on the finite interval $[0,1]$ which is divided in
$m$ equal subintervals $K$ of length $\triangle x=1/m$. The finite-volume method is
employed and the generic unknown $u_K$ is an approximation of the integral
$\int_K udx$. The difference quotient 
$\DD u_{K,\sigma}/(\triangle x):=(u_{K,\sigma}-u_K)/(\triangle x)$ approximates
the gradient of $u$ in the subinterval $K$, where $u_{K,\sigma}$ is the value
in the neighboring element $K'$ such that 
$\overline{K}\cap\overline{K'}=\{\sigma\}$. Then the
flux $J_u=-(\na(uT)+u\na V)$ through the point $\sigma$ can be approximated
by 
\begin{equation}\label{6.J}
  J_{u,K,\sigma}=-\frac{1}{\triangle x}\bigg(\DD(uT)_{K,\sigma}
	+\frac12(u_K+u_{K,\sigma})\DD V_{K,\sigma}\bigg).
\end{equation}
Special care has to be taken for the discretization of the Joule heating term
$\J_n\cdot\na V$. We suggest to approximate it according to
$$
  \int_K \J_n\cdot\na V dx \approx
	\frac{1}{2\triangle x}\sum_{\sigma} \triangle x \J_{n,K,\sigma}\DD V_{K,\sigma},
$$
where the sum is (here and in the following) 
over the two end points of the interval $K$. 
The values $C_K$, $\vec{\Omega}_K$, $p_K$ are given by the
integrals of $C(x)$, $\vec{\Omega}(x)$, $p(x)$ over $K$, respectively, and
the values $\vec{\Omega}_\sigma$, $p_\sigma$ are the arithmetic averages of 
$\vec\Omega$, $p$ in the neighboring subintervals of the intersecting point 
$\sigma$, respectively. Finally, we set
$\eta_\sigma=\sqrt{1-p_\sigma^2}$.

The stationary solution is computed as the limit $t_k=k\triangle t\to\infty$ from
the implicit Euler finite-volume discretization of 
\eqref{1.et1a.n0}-\eqref{1.et1a.vecn}. We solve first the Poisson equation
for the electric potential $V^k$ with the charge density from the previous time
step $k-1$, solve then the moment equations for $(n_0^k,W_0^k,\vec{n}^k)$,
and update finally the temperature. Given 
$(n_{0,K}^{k-1},\vec{n}_K^{k-1},T_K^{k-1})$ and 
$W_{0,K}^{k-1}=\frac32 n_{0,K}^{k-1}T_K^{k-1}$, the numerical scheme reads as
\begin{align*}
  & -\frac{\lambda_D^2}{\triangle x}\sum_\sigma \DD V_{K,\sigma}^k 
	= \triangle x(n_{0,K}^{k-1}-C_K), \\
	& \frac{\triangle x}{\triangle t}(n_{0,K}^k - n_{0,K}^{k-1})
	+ \sum_\sigma \J_{n,K,\sigma}^k = 0, \\
		& \frac{\triangle x}{\triangle t}(W_{0,K}^{k}-W_{0,K}^{k-1})
  + \sum_\sigma \J_{W,K,\sigma}^k + \frac{1}{2\triangle x}\sum_\sigma\triangle x
	\J_{n,K,\sigma}^k \DD V_{K,\sigma}^k = 0, \\
	& \frac{\triangle x}{\triangle t}(\vec{n}_K^k - \vec{n}_K^{k-1})
	+ \sum_\sigma \vec{\J}_{K,\sigma}^k 
	+ \gamma\triangle x(\vec{\Omega}_K\times\vec{n}_K^k)
	= -\frac{\triangle x}{\tau_{\rm sf}}\vec{n}_K^k, \\
	& T_K^k = \frac23 \frac{W_{0,K}^{k}}{n_{0,K}^k},
\end{align*}
and the discrete fluxes are defined by
\begin{align*}
  \J_{n,K,\sigma}^k &= -D_0\eta^{-2}_\sigma\big(J_{n,K,\sigma}^k 
	- p_\sigma\vec{\Omega}_\sigma\cdot\vec{J}_{n,K,\sigma}^k\big), \\
	\J_{W,K,\sigma}^k &= -\frac53 D_0\eta_\sigma^{-2}\big(J_{W,K,\sigma}^k
	- p\vec{\Omega}_\sigma\cdot\vec{J}_{W,K,\sigma}\big), \\
	\vec\J_{K,\sigma}^k &= -D_0\eta_\sigma^{-2}
	\big(-p_\sigma\vec{\Omega}_\sigma J_{n,K,\sigma}^k
	+ (1-\eta_\sigma)\vec{\Omega}_\sigma\otimes\vec{\Omega}_\sigma
	\cdot\vec{J}_{n,K,\sigma}^k + \eta_\sigma\vec{J}_{n,k,\sigma}^k\big), 
\end{align*}
and the fluxes $J_{n,K,\sigma}^k$, $\J_{W,K,\sigma}^k$, and
$\vec{J}_{K,\sigma}^k$ are discretized according to \eqref{6.J} with the
exception that the temperature and the densities in the drift term are 
explicit, i.e.
$$
  J_{u,K,\sigma}^k = -\frac{1}{\triangle x}\bigg(\DD(u^kT^{k-1})_{K,\sigma}
	+\frac12(u_K^{k-1}+u^{k-1}_{K,\sigma})\DD V^k_{K,\sigma}\bigg).
$$
Note that we have introduced the scaled diffusion coefficient $D_0$ and the
parameter $\gamma$, which come from the scaling of the equations. 
The values are $D_0\approx 6.9\cdot 10^{-4}$ and $\gamma=4$.
The scaled Debye length equals $\lambda_D\approx 1.2\cdot 10^{-4}$.
We have chosen the (scaled) boundary conditions 
$n_0=1$, $\vec{n}=0$, and $V=V_D$ at $x=0,1$ with $V_D(0)=0$ and $V_D(1)=U/U_T$.
Here, $U_T=0.026\,$V is the thermal voltage at room temperature.

The discrete linear system is solved for each time step $k$ until the maximum norm
of the difference between two consecutive solutions is smaller than a predefined
threshold ($10^{-8}\ldots 10^{-10}$). This solution is considered as a steady state.
The numerical parameters are $\triangle x=0.003$, $\triangle t=5\cdot 10^{-4}\ldots
10^{-3}$, and the (unscaled) physical parameters are $D=10^{-3}\,$m$^2$s$^{-1}$
(diffusion coefficient), $\tau_{\rm sf}=10^{-12}\,$s, and $U=-1\,$V (applied
bias). 


\subsection{Three-layer structure}\label{sec.3}

As the first numerical experiment, we consider a three-layer structure
which consists of a nonmagnetic layer sandwiched between two ferromagnetic layers;
see Figure \ref{fig.3}. This structure may be regarded as a diode with ferromagnetic 
source and drain regions. The length of the diode is $L=1.2\,\mu$m, the
ferromagnetic layers have length $\ell=0.2\,\mu$m, and the doping concentrations
are $C=10^{23}$\,m$^{-3}$ in the highly doped regions and 
$C=4\cdot 10^{20}$\,m$^{-3}$ in the lowly doped region.

\begin{figure}[ht]
\centering\includegraphics[width=0.6\textwidth]{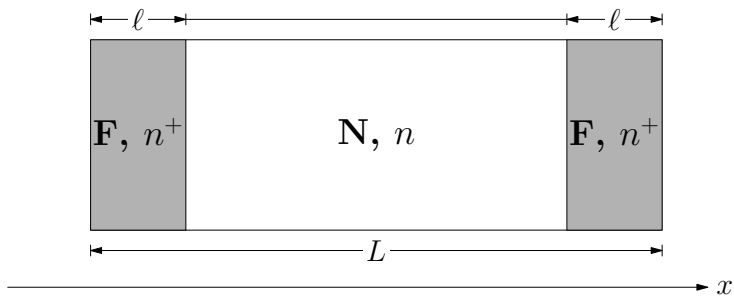}
\caption{Geometry of the three-layer structure with ferromagnetic ({\bf F}) 
highly doped ($n^+$) source and drain regions and nonmagnetic ({\bf N}) 
lowly doped ($n$) channel region.}
\label{fig.3}
\end{figure}

The local magnetization in the side regions is aligned with the $z$-axis
(orthogonal to the diode), $\vec{\Omega}(x)=(0,0,1)^\top$ for 
$x\in[0,\ell]\cup[L-\ell,L]$ and $\vec\Omega(x)=0$ else. 
The polarization in the ferromagnetic regions equals $p=0.66$. 

Figure \ref{fig.3n} shows the stationary charge density $n_0$ (left panel)
and the spin density $\vec{n}=(0,0,n_3)$ (right panel), 
compared with the solution to the corresponding
spinorial drift-diffusion model (with constant temperature). As expected, the
charge densities are similar with some small differences close to the junction of
the drain region. The spin component $n_3$ exhibits some peaks around the
junctions which can be explained by the discontinuity of $p(x)$ (and hence 
$\eta(x)$) at the junctions \cite[Sec.~8.1]{PoNe11}. The peaks are smaller
in the energy-transport model which may be due to thermal diffusion. 

\begin{figure}[ht]
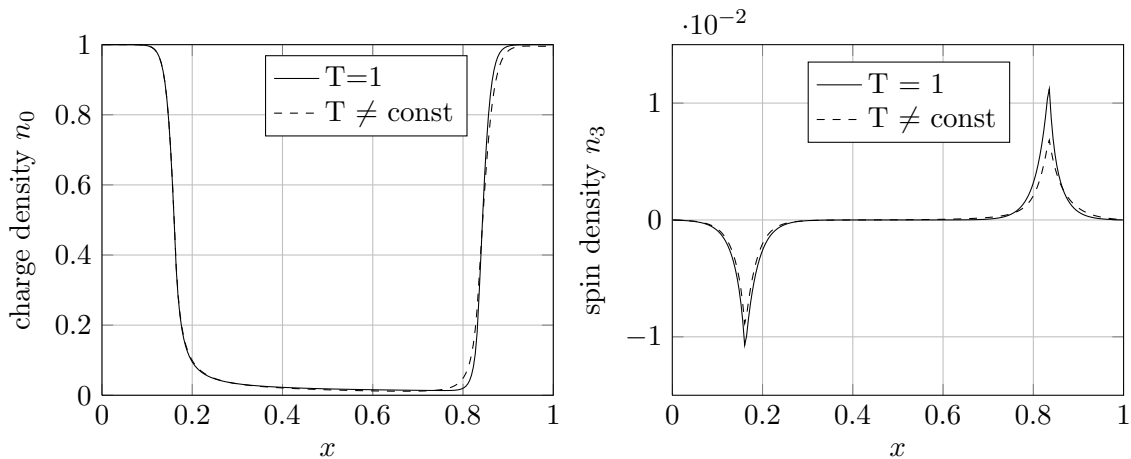

\centering\small
\input{n0_fnf.tikz}\input{n3_fnf.tikz}
\caption{Charge density $n_0$ (left) and spin density $n_3$ (right) in the
three-layer structure computed from the spin energy-transport model 
($T\neq\mbox{const.}$) and from the corresponding spin drift-diffusion model
($T=1$).}
\label{fig.3n}
\end{figure}

The temperature for different values of the polarization $p$ is illustrated
in Figure \ref{fig.3T}. The case $p=0$ corresponds to a nonmagnetic diode.
The temperature maximum increases with $p$ but the temperature decreases with $p$
in the drain region. Possibly, higher values of $p$ lead to stronger heat fluxes
increasing the temperature in the channel region.

\begin{figure}[ht]
\centering\small\input{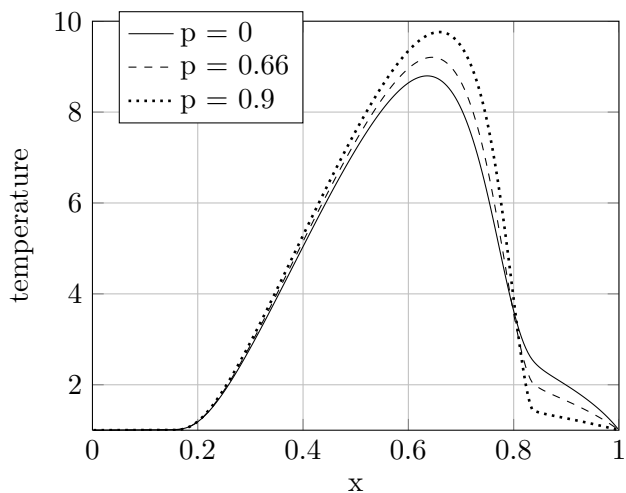}
\caption{Temperature in the three-layer structure for various 
polarizations $p$.}
\label{fig.3T}
\end{figure}


\subsection{Five-layer structure}

The five-lyer structure is composed of two ferromagnetic layers sandwiched between
two nonmagnetic layers and separated by a thin nonmagnetic layer in the middle
of the structure; see Figure \ref{fig.5}. 
The choice of the lengths $L$ and $\ell$ and of the doping
concentrations is as in Subsection \ref{sec.3}. The middle region has the
thickness $d=L/21\approx 60\,$nm. Again we take $p=0.66$. The local
magnetization is different in the two layers:
$\vec\Omega(x)=(0,0,1)^\top$ for $x\in[L/6,10L/21]$, 
$\vec\Omega(x)=(0,1,0)^\top$ for $x\in[11L/21,5L/6]$, and
$\vec\Omega(x)=0$ else. 

\begin{figure}[ht]
\includegraphics[width=0.6\textwidth]{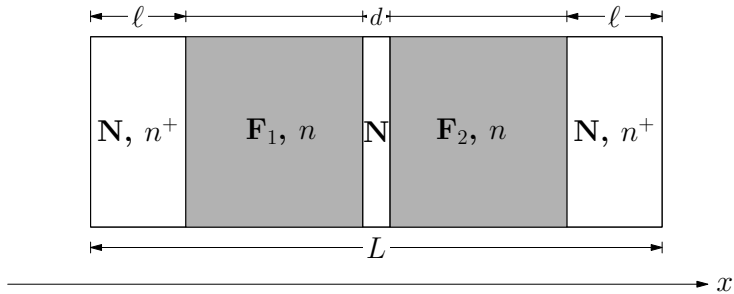}
\caption{Geometry of the five-layer structure with 
ferromagnetic ({\bf F}$_1$, {\bf F}$_2$) lowly doped ($n$) regions and
nonmagnetic ({\bf N}) regions. The source and drain regions are highly doped $(n^*$),
while the middle region is lowly doped.}
\label{fig.5}
\end{figure}

The effect of the temperature is now stronger than in the three-layer structure.
The charge density $n_0$ and temperature $T$ are presented
in Figure \ref{fig.5nT}. The interplay of the charge and spin densities in the
nonmagnetic middle region causes a small hump in $n_0$
and a more significant increase before the drain junction, compared to
Figure \ref{fig.3n} (left). The hump is larger when the electric potential is
a linear function and the temperature is constant; see Figure 3 in \cite{PoNe11}.
The temperature maximum decreases with $p$,
opposite to the situation in the three-layer structure. We observe that
the polarization strongly influences the temperature.
When $p=0$, we obtain the same curve as in Figure \ref{fig.3T} since
this describes the same nonmagnetic diode.

\begin{figure}[ht]
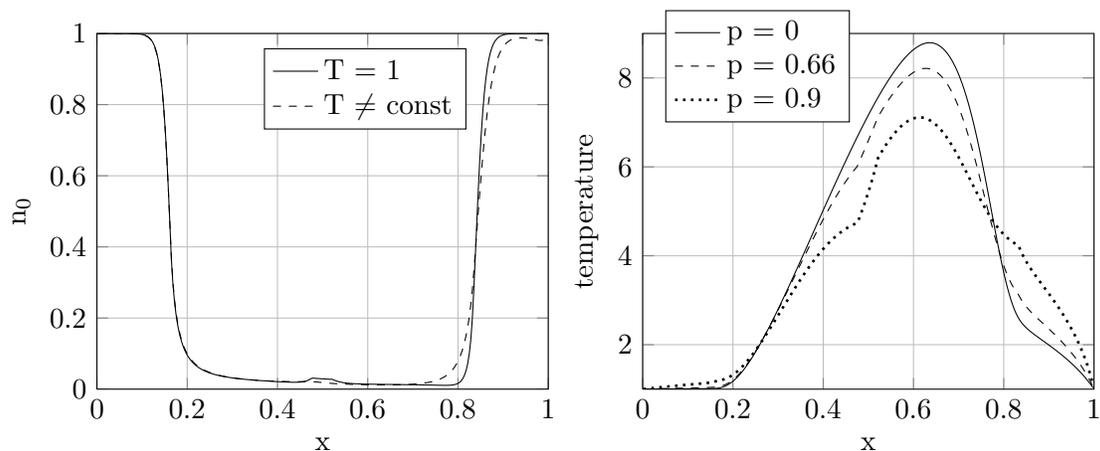

\centering\small
\input{n0_5layers.tikz}\input{temper_5layers_3.tikz}
\caption{Charge density $n_0$ (left) and temperature $T$ (right) 
in the five-layer structure.}
\label{fig.5nT}
\end{figure}

In contrast to the three-layer structure, all components of the spin vector
density are nonzero. However, the component $n_1$ is relatively small. We
present the remaining components $n_2$ and $n_3$ in Figure \ref{fig.5n23}.
The temperature causes a significant smoothing of the peaks between the
magnetic/nonmagnetic junctions.

\begin{figure}[ht]
\centering\small\input{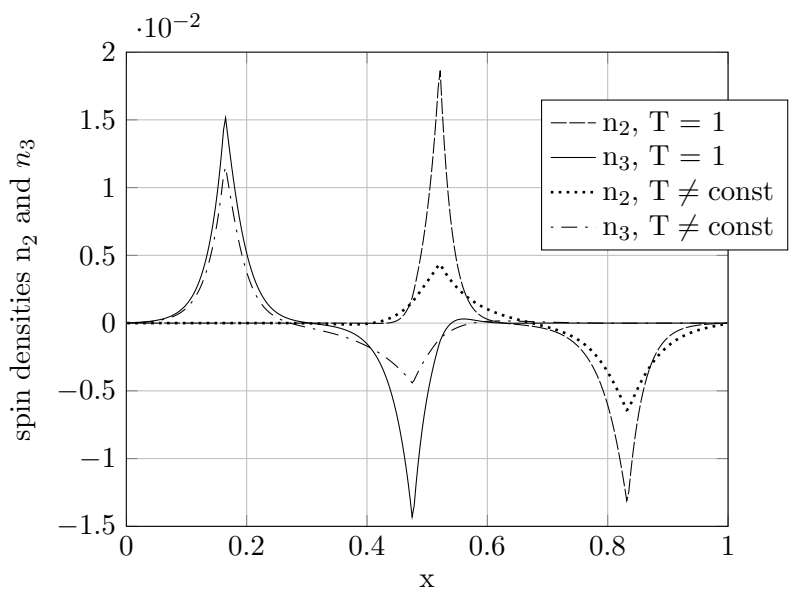}
\caption{Spin density components $n_2$ and $n_3$ in the
five-layer structure computed from the spin energy-transport model 
($T\neq\mbox{const.}$) and from the corresponding spin drift-diffusion model
($T=1$).}
\label{fig.5n23}
\end{figure}


\end{document}